\documentclass[11pt,a4paper]{article}

\usepackage[T1]{fontenc}
\usepackage{lmodern}
\usepackage[left=27mm,right=27mm,top=25mm,bottom=26mm,
            headheight=12pt,headsep=14pt,footskip=25pt]{geometry}
\usepackage{amsmath,amssymb,amsthm,mathtools,aliascnt}
\usepackage{microtype,booktabs,array,longtable,tabularx,needspace}
\usepackage{fancyhdr,titlesec}
\usepackage{tikz}
\usetikzlibrary{arrows.meta}
\usepackage[font=small,labelfont=bf,labelsep=period,skip=7pt]{caption}
\usepackage[hidelinks]{hyperref}
\usepackage[capitalize,nameinlink,noabbrev]{cleveref}
\allowdisplaybreaks[2]
\titleformat{\section}{\normalfont\large\bfseries}{\thesection}{0.7em}{}
\titlespacing*{\section}{0pt}{2.6ex plus .6ex minus .3ex}{1.1ex plus .2ex}
\titleformat{\subsection}{\normalfont\normalsize\bfseries}{\thesubsection}{0.65em}{}
\titlespacing*{\subsection}{0pt}{2.1ex plus .5ex minus .2ex}{0.8ex plus .2ex}
\fancypagestyle{plain}{\fancyhf{}\fancyfoot[C]{\small\thepage}}
\makeatletter
\renewcommand{\maketitle}{%
  \thispagestyle{plain}%
  \begin{center}
    {\fontsize{18}{23}\selectfont\bfseries\@title\par}
    \vspace{12pt}
    {\normalsize\begin{tabular}[t]{c}\@author\end{tabular}\par}
  \end{center}
  \vspace{4pt}}
\makeatother
\numberwithin{equation}{section}
\newtheorem{theorem}{Theorem}[section]
\newaliascnt{lemma}{theorem}
\newtheorem{lemma}[lemma]{Lemma}
\aliascntresetthe{lemma}
\newaliascnt{proposition}{theorem}
\newtheorem{proposition}[proposition]{Proposition}
\aliascntresetthe{proposition}
\newaliascnt{corollary}{theorem}
\newtheorem{corollary}[corollary]{Corollary}
\aliascntresetthe{corollary}
\theoremstyle{definition}
\newaliascnt{definition}{theorem}

\aliascntresetthe{definition}
\theoremstyle{remark}
\newaliascnt{remark}{theorem}

\aliascntresetthe{remark}
\newaliascnt{example}{theorem}

\aliascntresetthe{example}
\newcommand{\R}{\mathbb R}
\newcommand{\E}{\mathbb E}

\DeclareMathOperator{\diag}{diag}
\DeclareMathOperator{\Tr}{Tr}
\DeclareMathOperator{\disc}{disc}
\DeclareMathOperator{\supp}{supp}
\DeclareMathOperator{\rank}{rank}
\DeclareMathOperator{\range}{range}
\newcommand{\one}{\mathbf 1}
\newcommand{\cV}{\mathcal V}
\newcommand{\cD}{\mathcal D}
\newcommand{\cE}{\mathcal E}
\newcommand{\cH}{\mathcal H}
\newcommand{\cF}{\mathcal F}
\newcommand{\cP}{\mathcal P}
\newcommand{\cC}{\mathcal C}
\newcommand{\cR}{\mathcal R}

\newcommand{\norm}[1]{\lVert#1\rVert}
\newcommand{\ip}[2]{\langle#1,#2\rangle}
\newcommand{\logplus}{\log_+}
\hypersetup{pdftitle={A (log n)\textasciicircum(1/4) Bound for the Koml\'os Problem},
 pdfauthor={Eren Ercan},pdflang={en-GB},pdfsubject={Combinatorial discrepancy and affine spectral independence}}

\newcommand{\cT}{\mathcal T}

\title{A $({\log n})^{1/4}$ Bound for the Koml\'os Problem}
\author{Eren Ercan\\[4pt]{\small\href{mailto:eren321@gmail.com}{\texttt{eren321@gmail.com}}}}
\date{}
\begin{document}
\maketitle
\begin{abstract}
Let $A\in\R^{m\times n}$ have columns of Euclidean norm at most one.
We prove that
\[
 \disc(A)\le2395\left(1+\log_+\frac n9\right)^{1/4}+2\sqrt2.
\]
Here $\log_+t=\max\{0,\log t\}$. Building on Bansal and Jiang's affine
spectral independence framework, we remove the $(\log\log n)^{7/4}$
factor from their bound. The fourth root comes
from balancing the logarithmic decrease in the alive dimension against
the fourth power of the row thresholds. Historical exponential sums control the
covariance budget across size classes with summable thresholds.
An exact threshold-sum certificate gives the
coefficient $2395$, and rounding at most eight remaining fractional
coordinates costs $2\sqrt2$. The finite construction also gives partial
colourings from any prescribed starting point and at any prescribed
depth, preserving existing signs. We formalize the partial- and
full-colouring theorems in Lean, including the finite trajectory, exact
threshold sum and final rounding, with Bansal--Jiang Theorem~A.4 as the
sole external research theorem assumption.
\end{abstract}
\section{Introduction and results}\label{sec:introduction}
For a real matrix $A\in\R^{m\times n}$, its combinatorial discrepancy is
\[
 \disc(A)=\min_{\sigma\in\{-1,1\}^n}\norm{A\sigma}_\infty.
\]
The Koml\'os conjecture asserts that this quantity is bounded by an
absolute constant when every column of $A$ has Euclidean norm at most one.
Banaszczyk established an $O(\sqrt{\log n})$ upper bound~\cite{Ban98}.
Bansal and Jiang introduced decoupling via affine spectral
independence~\cite{BJ25}. Their bound, in the form given in their
subsequent exposition, is
$O((\log n)^{1/4}(\log\log n)^{7/4})$~\cite{BJ26}.
Building on their framework and covariance theorem, we prove the
following bound.

\begin{theorem}\label{thm:main}
Let $m,n\ge1$ and let $A\in\R^{m\times n}$ satisfy
$\sum_{i=1}^m A_{ij}^2\le1$ for every $j\in[n]$. Then
\begin{equation}\label{eq:main}
 \disc(A)\le2395\left(1+\logplus\frac n9\right)^{1/4}+2\sqrt2,
 \qquad \logplus u=\max\{0,\log u\}\quad(u>0).
\end{equation}
\end{theorem}

\begin{samepage}
The construction starts from any given fractional colouring. A coordinate
of a point in $[-1,1]^n$ is \emph{fractional} when its absolute value is
strictly less than one. All logarithms are natural.

\begin{theorem}[Prescribed-depth partial colouring]\label{thm:partial}
Let $A$ satisfy the hypotheses of \cref{thm:main}, let $b\in[-1,1]^n$,
and let $q$ be the number of fractional coordinates of $b$. For every
$R\ge0$, there is $y\in[-1,1]^n$ such that $y_j=b_j$ whenever
$|b_j|=1$, at most
\begin{equation}\label{eq:target}
 N_R(q):=\max\{\lceil qe^{-R}\rceil-1,8\}
\end{equation}
coordinates of $y$ are fractional, and
\begin{equation}\label{eq:partial}
 \norm{A(y-b)}_\infty\le2395(1+R)^{1/4}.
\end{equation}
For $q>8$, the endpoint is attained by a finite construction at a positive
mesh, using a feasible covariance and a finite direction law at each step.
For $q\le8$, one may take $y=b$.
\end{theorem}
\end{samepage}

The depth $R$ is prescribed before the trajectory is constructed. The
thresholds used during that trajectory remain fixed. Thus the quantifiers
in \cref{thm:partial} permit a different trajectory for each depth.

\begin{corollary}[Completion from a given point]\label{cor:completion}
Under the same matrix hypotheses, every $b\in[-1,1]^n$ with $q\ge1$
fractional coordinates admits a signing $\sigma$ preserving its existing
signs and satisfying
\begin{equation}\label{eq:completion}
 \norm{A(\sigma-b)}_\infty
 \le2395\left(1+\logplus\frac q9\right)^{1/4}+2\sqrt2.
\end{equation}
For $q\le8$, the bound $\sqrt q$ follows by direct Euclidean rounding;
for $q=0$, take $\sigma=b$.
\end{corollary}

For $q\ge9$, taking $R=\log(q/9)$ in \cref{thm:partial} leaves at most
eight fractional coordinates. Euclidean rounding then costs at most
$\sqrt8=2\sqrt2$; see \cref{sec:completion}. Taking $b=0$ gives
\cref{thm:main}, hence an $O((1+\log n)^{1/4})$ bound independent of
the number of rows. The construction uses exact real covariance choices
and comparisons; its running time in a bit model is not estimated here.

\subsection{Relation to the Bansal--Jiang construction}
We follow the notation and regularized-discrepancy framework of
Bansal and Jiang's original paper~\cite{BJ25}. The walk guided by a semidefinite
program (SDP), the blocking of the current point, and the exact squared-norm
identity come from their original construction~\cite[Section~2.1]{BJ25}.
Their decoupling argument controls a sum of capped exponentials
using affine spectral independence and negative drift~\cite[Section~3.2.5]{BJ25}.
Their Beck--Fiala argument
uses adaptive size classes~\cite[Section~6]{BJ25}. These ideas underlie
the construction here. The covariance theorem we invoke is their original
Theorem~A.4; \cref{sec:covariance} derives the finite-row form, also stated
in~\cite[Theorem~2.12]{BJ26}.

Our contribution is the treatment of successive restrictions of arbitrary
real rows. Each restriction keeps its coefficients and birth point while
it is active. At retirement, its completed contribution is retained and
its surviving restriction is processed at smaller size classes. We adapt the capped-exponential sum to this history, with column weights,
and choose complete updates that preserve a single exponential potential bound. The resulting bound controls
the number of dangerous restrictions relative to the current number of
alive coordinates. Together with the covariance theorem, it ensures
feasibility throughout the selected trajectory.

In the Koml\'os construction of~\cite[Section~4.1]{BJ25}, also described
in the subsequent exposition~\cite[Section~3.2]{BJ26}, the input is
decomposed by entry-magnitude
scales. Here a size class records the squared mass of a current row
restriction. The coefficient cap imposed within that class permits a
weighted column estimate for all real entries. The thresholds over the
successive size classes are summable. We prove the required weighted
exponential estimate directly, including its finite-step remainder.
Section~\ref{sec:local} establishes this estimate for complete updates,
including births and retirements, and \cref{sec:trajectory} uses it to
construct the finite trajectory.

\subsection{Overview of the proof}\label{sec:overview}
The construction evolves the $q$ initially fractional coordinates. Write
$x_t$ for the current colouring, $\cV_t$ for its alive set, and
$n_t=|\cV_t|$. For a fixed mesh $\gamma>0$, a coordinate is alive when
$|x_t(j)|<1-\gamma$. Once a coordinate ceases to be alive, it is
held fixed until the final rounding to its sign. To \emph{block} a vector
means to require every supported direction to be orthogonal to it.
The update is
\[
 x_{t+1}=x_t+\gamma v_t,
 \qquad \norm{v_t}_2=1,\quad
 \supp v_t\subseteq\cV_t,\quad \ip{x_t}{v_t}=0.
\]
Blocking the current point gives
$\norm{x_{t+1}}_2^2=\norm{x_t}_2^2+\gamma^2$. Consequently, any rule
that chooses an admissible step at every nonterminal state reaches its
stopping target after finitely many steps. The issue is to
preserve the covariance budget while the alive set changes.
All expectations are taken under the
finite law chosen at the current state. We use that law to select a
successor with a suitable potential bound, and then choose a new law
at the selected state.

\paragraph{Successive restrictions and size classes.}
For $c\in\R^q$, put
$r_t(c)=\sum_{j\in\cV_t}c(j)^2$. Large rows are blocked, so they incur
zero discrepancy until their size is bounded by a fixed constant $M$.
For $s_k=M2^{-k}$, a class-$k$ restriction has current size between
$s_k/2$ and $s_k$. Given a target $B_k$, remove the coordinates with
$|c(j)|>s_k/B_k$. Their total $\ell_1$ norm is at most $B_k$, so
their future discrepancy costs at most $2B_k$. The retained coefficients
satisfy $\norm{c}_\infty\le s_k/B_k$.

The retained vector is held fixed until its alive squared mass falls to
at most $s_k/2$. At that point we retain its completed discrepancy and process
its surviving restriction at later classes. A row contributes at most
one peeled vector and one tracked interval at each class. Successive
restrictions may have common coordinates. The exact row identity in
\cref{sec:row-identity} adds their contributions over the corresponding
time intervals.

\paragraph{Regularization makes blocking compatible with negative drift.}
If a class-$k$ restriction $c$ is born at $u$, define, for each sign
$\varepsilon\in\{-1,1\}$,
\[
 Z_c^\varepsilon(x)=\varepsilon\ip{c}{x-u}
 +\beta_k\sum_jc(j)^2\bigl(u(j)^2-x(j)^2\bigr),
 \qquad \beta_k=\frac{B_k}{2s_k}.
\]
This is the Bansal--Jiang regularized discrepancy measured from the
birth point. Its moving gradient is
$e_c^\varepsilon(x)=\varepsilon c-2\beta_k c^{\odot2}\odot x$.
The exact increment along a direction $v$ is
\[
 \gamma\ip{e_c^\varepsilon(x)}{v}
 -\beta_k\gamma^2\sum_jc(j)^2v(j)^2.
\]
A centered direction law removes the linear term in expectation. If
$Z_c^\varepsilon$ reaches $B_k$ and the restriction remains active,
we block its moving gradient on subsequent steps, removing the linear
term for every supported direction. A simultaneous retirement ends the
record without adding a blocking constraint. The coefficient
cap bounds the gradient and the possible threshold overshoot.

\paragraph{Weighted exponential sums control dangerous restrictions.}
For each signed restriction, an exponential $X$ starts at one, follows
$\exp(\lambda_k Z_c^\varepsilon)$ while live, and is capped at
$e^{L_k}$, where $\lambda_kB_k=L_k$. After a threshold crossing or
retirement, retain its value. At column $j$, weight it by
$q_c^{(j)}=c(j)^2/s_k$. A row has at most one birth in a fixed class,
so the total historical birth mass at that column is at most $2/s_k$.
The factor two accounts for the two signs.

Let $\Phi^{(j,k)}$ denote the weighted exponential sum and let $I^{(j,k)}$ denote
total birth mass. The difference $G^{(j,k)}=\Phi^{(j,k)}-I^{(j,k)}$
is unchanged by a birth. Every current dangerous signed restriction
contributes its capped exponential and has alive normalized mass greater
than $1/2$. A bound on $G$ therefore bounds the number of such
restrictions, including after arbitrarily many earlier retirements.

The construction preserves
\[
 \Psi_t=\sum_{k,j}w_k\exp\bigl(G_t^{(j,k)}/K_k\bigr)\le q,
 \qquad \sum_{k\ge0}w_k=1.
\]
Jensen's inequality over the current alive columns gives a term
$\log(q/n_t)$. While a further step is required, the integer alive
count satisfies $n_t>N_R(q)$, hence
$n_t\ge\lceil qe^{-R}\rceil\ge qe^{-R}$. Thus $\log(q/n_t)\le R$.
The resulting dangerous-family count is small
enough for the covariance theorem. Its dependence on the number of
columns comes from this single aggregate potential.

\Cref{fig:restriction-history} shows how a threshold crossing and a
retirement affect different parts of the state.
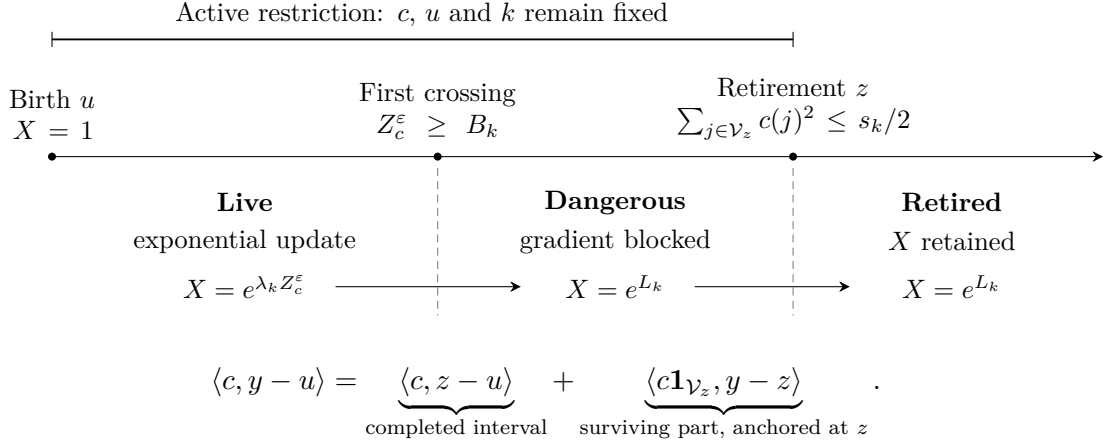
\begin{figure}[tbp]
\centering
\begin{tikzpicture}[x=1cm,y=1cm,
  every node/.style={font=\small,inner sep=2pt,align=center},
  >={Stealth[length=4pt,width=4pt]}]
  \draw[line width=.7pt] (1,1.55) -- (10.8,1.55);
  \draw (1,1.45) -- (1,1.65) (10.8,1.45) -- (10.8,1.65);
  \node[above,text width=9.2cm] at (5.9,1.65)
    {Active restriction: $c$, $u$ and $k$ remain fixed};

  \draw[->] (1,0) -- (14.9,0);
  \foreach \x in {1,6.1,10.8}{\fill (\x,0) circle (1.5pt);}
  \node[above,text width=1.8cm] at (1,.15) {Birth $u$\\$X=1$};
  \node[above,text width=3.8cm] at (6.1,.15)
    {First crossing\\$Z_c^\varepsilon\ge B_k$};
  \node[above,text width=3.8cm] at (10.8,.15)
    {Retirement $z$\\$\sum_{j\in\cV_z}c(j)^2\le s_k/2$};

  \draw[densely dashed,gray] (6.1,-.15) -- (6.1,-2.1);
  \draw[densely dashed,gray] (10.8,-.15) -- (10.8,-2.1);
  \node at (3.55,-.60) {\textbf{Live}};
  \node at (8.45,-.60) {\textbf{Dangerous}};
  \node at (12.9,-.60) {\textbf{Retired}};
  \node at (3.55,-1.12) {exponential update};
  \node at (8.45,-1.12) {gradient blocked};
  \node at (12.9,-1.12) {$X$ retained};
  \node at (3.55,-1.72) {$X=e^{\lambda_k Z_c^\varepsilon}$};
  \node at (8.45,-1.72) {$X=e^{L_k}$};
  \node at (12.9,-1.72) {$X=e^{L_k}$};
  \draw[->] (4.75,-1.72) -- (7.2,-1.72);
  \draw[->] (9.5,-1.72) -- (11.55,-1.72);
\end{tikzpicture}
\vspace{3pt}
\[
 \ip{c}{y-u}
 =\underbrace{\ip{c}{z-u}}_{\text{completed interval}}
 +\underbrace{\ip{c\one_{\cV_z}}{y-z}}_{\text{surviving part, anchored at }z}.
\]
\caption{Restriction lifetime and retained history. The exponential $X$ is
shown for a signed record that crosses its threshold before retirement.
The dangerous mark persists even if $Z_c^\varepsilon$ later decreases.
The other sign has its own status; both retire together. A sign that
remains live receives its final exponential update before retirement.
In the identity, $\cV_z$ is
the alive set at $z$, and $y_j=z_j$ outside it. The contributions concern
successive time intervals even when their coefficient supports overlap.
The surviving part is normalized at later classes, with new anchor $z$.}
\label{fig:restriction-history}
\end{figure}

\Needspace{10\baselineskip}
\paragraph{The smooth estimate for a complete update.}
The update of the point determines which coordinates freeze, which
restrictions retire, and which new restrictions are born. The coefficients
of new restrictions can depend on the selected direction. Even so, the
increment of $G$ is bounded by a
smooth expression using only old live records. The initial exponential sum cancels the
birth mass, stopped exponentials remain fixed, and the last update of
a retiring live record is included before retirement. These identities
justify differentiating the old-state expression and then applying its
bound to every complete successor.

The second derivative of the outer exponential contains a covariance
term for the sum of the signed responses. Affine spectral independence
controls precisely that term. A symmetric direction law cancels odd
Taylor terms. The buffer at the cube boundary provides negative
curvature of order $\gamma$; its contribution to one step is of order
$\gamma^3$. The fourth-order remainder is bounded by $\gamma^4$ times
the same nonnegative weighted energy. This permits a uniformly small
positive mesh even at states where that energy is arbitrarily small.

\paragraph{The schedule explains the exponent.}
The column coefficient cap contributes a factor $B_k^{-2}$ to the
weighted covariance bound. The ratio of the exponential parameter to
the negative-drift parameter contributes another factor $B_k^{-2}$.
Thus the dangerous-family estimate is controlled by
$s_kL_k(L_k+R)/B_k^4$. We allocate a geometrically decreasing fraction
$\delta_k$ of the available blocking dimension to class $k$ and take
\[
 B_k\asymp
 \left(\frac{s_kL_k(L_k+R)}{\delta_k}\right)^{1/4}
 +\sqrt{s_kL_k}.
\]
The ratio $s_k/\delta_k$ decreases geometrically, whereas $L_k$ grows
linearly with $k$. Hence the thresholds sum to $O((1+R)^{1/4})$.
The row identity converts this sum into a discrepancy bound. A strict
numerical reserve pays for the finite-mesh overshoot and buffered rounding.

\subsection{Notation and formalization}
Rows of $A$ are $a_1,\ldots,a_m$, and $a_i(j)=A_{ij}$, following
\cite{BJ25}. For a vector $c$, $c^{\odot2}$ denotes its
coordinatewise square and $c\odot x$ denotes coordinatewise multiplication.
We write $h=|V|$ for the number of alive coordinates at a fixed state
and $W$ for the blocking subspace, following the SDP notation
of~\cite{BJ25}. For a symmetric matrix $U$, $\diag U$ is the diagonal matrix with the
same diagonal. The relation $U\preceq V$ is the positive-semidefinite
order. In our restatement of Theorem~A.4, $\eta_0$ is the coordinate
parameter of~\cite{BJ25}, so $U\preceq\eta_0^{-1}\diag U$.
Elsewhere we abbreviate its reciprocal by $\eta=\eta_0^{-1}$.
Two signs are tracked for each restriction;
the input matrix itself always retains its original column normalization.
Our step $x^+=x+\gamma v$ has $dt=\gamma^2$ in the time-step
notation of~\cite{BJ25}. We write $Q=U/\Tr U$ for the covariance of
the unit direction law. Thus $\E[(x^+-x)(x^+-x)^{\mathsf T}]
=\gamma^2Q$, as in the normalized update of~\cite[Section~2.1]{BJ25}.

We formalize the partial- and full-colouring theorems in Lean, including
the historical trajectory, the exact threshold sum and final rounding.
Bansal--Jiang Theorem~A.4 is the sole external research theorem
assumption. \Cref{app:formalization} gives the formal statements and
their correspondence with the written proof.

\section{SDP feasibility and the direction law}\label{sec:covariance}
For a centered random vector $v$ with covariance $Q$, the inequality
$Q\preceq\eta\diag Q$ says that
\[
 \E\ip{a}{v}^2\le\eta\sum_j a(j)^2\E v(j)^2
 \qquad(a\in\R^h).
\]
For a matrix $E$ of row responses, the inequality
$EQE^{\mathsf T}\preceq\alpha\diag(EQE^{\mathsf T})$
gives the analogous bound for all linear combinations of those responses:
\[
 \E\left(\sum_a z_a\ip{E(a,\cdot)}{v}\right)^2
 \le\alpha\sum_a z_a^2\E\ip{E(a,\cdot)}{v}^2.
\]
These are the spectral independence and affine spectral independence
conditions used by Bansal and Jiang. Later, $z_a$ will be the weighted
exponential of a live signed restriction, so this inequality controls the
second moment of their total response.

\begin{theorem}[Bansal--Jiang, Theorem A.4 of~\cite{BJ25}]
\label{thm:source}
Let $h\ge1$, let $W\subseteq\R^h$ have dimension $dh$, and let
$\widetilde E_s$ have $r_sh$ rows with $r_s\ge1$, for finitely many
indices $s$. Suppose $0<\eta_0,\kappa,\eta_s<1$ and
\[
 d+\eta_0+\kappa+\sum_s\eta_s\le1.
\]
There is a positive-semidefinite matrix $U$ such that
\begin{gather*}
 w^{\mathsf T}Uw=0\quad(w\in W),\qquad U(j,j)\le1,\qquad
 \Tr U\ge\kappa h,\qquad U\preceq\eta_0^{-1}\diag U,\\
 \widetilde E_sU\widetilde E_s^{\mathsf T}
 \preceq\frac{r_s}{\eta_s}
       \diag(\widetilde E_sU\widetilde E_s^{\mathsf T})
 \qquad\text{for every }s.
\end{gather*}
\end{theorem}

The row ratios in this statement are at least one. The affine blocks
arising from a current size class may have fewer rows than alive
coordinates. Repeating rows within each block gives the following form.

\begin{corollary}[Finite-row covariance theorem]\label{cor:finite-row}
Let $h\ge1$ and let $W\subseteq\R^h$ satisfy $\dim W\le\Delta h$. Let
$E_s\in\R^{m_s\times h}$ be arbitrary finite matrices, let
$\eta>1$, $\kappa>0$, and $\alpha_s>0$, and suppose
\begin{equation}\label{eq:sdp-budget}
 \eta^{-1}+\kappa+\sum_s\frac{m_s}{\alpha_sh}<1-\Delta.
\end{equation}
Then there is $U\succeq0$ such that
\begin{align}
 UW&=0,& U(j,j)&\le1,& \Tr U&\ge\kappa h,
 &U&\preceq\eta\diag U,\label{eq:sdp-coordinate}\\
 E_sUE_s^{\mathsf T}&\preceq\alpha_s\diag(E_sUE_s^{\mathsf T})
 &&\text{for every }s.\label{eq:sdp-affine}
\end{align}
Here $UW=0$ means that $Uw=0$ for every $w\in W$.
\end{corollary}
\begin{proof}
Omit empty blocks. For each remaining block, put
\[
 t_s=\left\lceil\frac h{m_s}\right\rceil,\qquad
 r_s=\frac{t_sm_s}{h},\qquad
 \eta_s=\frac{m_s}{\alpha_sh},\qquad \eta_0=\eta^{-1}.
\]
Stack $t_s$ copies of $E_s$ into $\widetilde E_s$.
The strict budget puts all the source parameters in $(0,1)$ and permits
\cref{thm:source} with $d=\dim W/h$. Its affine coefficient is
$r_s/\eta_s=t_s\alpha_s$. Set $M_s=E_sUE_s^{\mathsf T}$ and test
that inequality on the block vector $(z,\ldots,z)$. The result is
\[
 t_s^2z^{\mathsf T}M_sz
 \le t_s^2\alpha_s z^{\mathsf T}(\diag M_s)z.
\]
Cancellation gives \eqref{eq:sdp-affine}. Positive semidefiniteness gives
$w^{\mathsf T}Uw=\norm{U^{1/2}w}_2^2$, so its vanishing implies
$Uw=0$.
\end{proof}

This reduction keeps the affine matrices separate. Its charge for a block
is exactly $m_s/(\alpha_sh)$. The finite-row formulation also appears
directly in~\cite[Theorem~2.12]{BJ26}, where arbitrary positive real row
ratios are allowed.

The coordinate and affine inequalities are homogeneous in $U$, so they
also hold for $Q=U/\Tr U$. The positive trace makes this normalization
well defined. The following spectral construction realizes $Q$ as the covariance
of a finite law. An annihilated vector is orthogonal to every atom, so
the blocking constraints hold for each possible successor.

\begin{lemma}[A symmetric unit direction law]\label{lem:direction-law}
Let $U\succeq0$ be nonzero and put $Q=U/\Tr U$. There is a finite
centrally symmetric law on $2\rank U$ positive-weight atoms $v$ such that
\begin{equation}\label{eq:direction-law}
 \E v=0,\qquad \E vv^{\mathsf T}=Q,\qquad
 \norm v_2=1,\qquad \norm v_\infty\le1.
\end{equation}
Every atom is in $\range U$, so it is orthogonal to every vector
annihilated by $U$.
\end{lemma}
\begin{proof}
Write the positive spectral decomposition
$U=\sum_{\ell=1}^{\rank U}\mu_\ell p_\ell p_\ell^{\mathsf T}$,
with orthonormal $p_\ell$ and $\mu_\ell>0$. Give each of
$p_\ell,-p_\ell$ probability $\mu_\ell/(2\Tr U)$. The two signs
cancel in the mean, and their covariance contributions sum to $Q$.
Each atom is a unit vector in the range of $U$. The kernel and range of a
symmetric positive-semidefinite matrix are orthogonal.
\end{proof}

The probability law is used for a finite average at the current state.
The trajectory will choose one atom whose complete successor has potential
at most that average. This choice is deterministic once an ordering of
the finite support and a feasible covariance have been fixed.

\begin{lemma}[Cube preservation and the squared-norm identity]\label{lem:clock}
Let $0<\gamma<1$, $x\in[-1,1]^q$, and
$V=\{j:|x(j)|<1-\gamma\}$. If $v$ is supported on $V$,
$\norm v_2=1$, and $\ip{x}{v}=0$, then $x^+=x+\gamma v$ belongs
to the cube, agrees with $x$ outside $V$, and satisfies
\begin{equation}\label{eq:clock}
 \norm{x^+}_2^2=\norm x_2^2+\gamma^2.
\end{equation}
Its alive set is a subset of $V$.
\end{lemma}
\begin{proof}
For $j\in V$, $|x^+(j)|\le|x(j)|+\gamma|v(j)|<1$, since
$|v(j)|\le1$. The other coordinates do not change. Expanding the squared
norm proves \eqref{eq:clock}. Every previously frozen coordinate retains
its value and hence remains frozen.
\end{proof}

\section{Successive restrictions of a real row}\label{sec:restrictions}
Fix an initial point $b$ with $q>8$ fractional coordinates and restrict
$A$ to those columns. Throughout the construction, $A$ has $q$ columns,
$b\in(-1,1)^q$, and discrepancy means the increment from $b$.
The original fixed signs will be restored at the end. Fix a positive mesh
$\gamma<1/2$, small enough that all coordinates of $b$ are initially
alive. Its final choice is made in \cref{sec:finite-attainment}.

For a state $x$, write
\[
 V=\{j:|x(j)|<1-\gamma\},\qquad h=|V|,\qquad
 r_x(c)=\sum_{j\in V}c(j)^2.
\]
The quantity $r_x(c)$ is the \emph{squared $\ell_2$ mass on alive coordinates},
called its \emph{size} below.
For a fixed vector $c$, it can decrease only when coordinates freeze.
The \emph{energy}
\begin{equation}\label{eq:energy}
 \cE_x(c)=\sum_jc(j)^2(1-x(j)^2)
\end{equation}
also depends on the positions of the coordinates within the cube. Its
increment contains the negative quadratic term in regularized discrepancy.
We denote energy by $\cE$ to distinguish it from the centered
exponential sum $G^{(j,k)}$ introduced in \cref{sec:capital}.

Set
\begin{equation}\label{eq:fixed-parameters}
 M=\frac{60}{7},\qquad s_k=M2^{-k},\qquad
 \eta=\frac{61}{18},\qquad \kappa=\frac1{65536},\qquad
 \alpha_k=\frac{1973}{100s_k}.
\end{equation}
Here $M$ is the large-row cutoff, $\eta$ is the coordinate spectral
independence coefficient, $\alpha_k$ is its affine counterpart in class
$k$, and $\kappa$ is the trace fraction reserved in the covariance theorem.
Their numerical values enter the budget in \cref{sec:schedule}.
Fix a finite cutoff $K$ satisfying $s_{K+1}\le1/q$. For example,
$K=q+3$ suffices, because $q\le2^q$ and $M<16$.
The construction uses classes $0,\ldots,K$. A residual of squared norm
at most $s_{K+1}$ is terminal. Its $\ell_1$ norm is at most one.

For now fix positive class thresholds $B_0,\ldots,B_K$. Their values,
and the positive exponential parameters $L_k$, will be chosen in
\cref{sec:schedule}. The geometric statements of this section hold for
any positive thresholds.

\subsection{Peeling and normalization}\label{sec:normalization}
A row $a_i$ is \emph{large} at $x$ while $r_x(a_i)>M$. Its alive
restriction is blocked at every such state. The column norm hypothesis
gives
\begin{equation}\label{eq:column-size}
 \sum_i r_x(a_i)=\sum_{j\in V}\sum_i a_i(j)^2\le h,
\end{equation}
so there are at most $h/M$ large rows. Once a row has size at most $M$,
we process its restriction to the alive set.

The coefficient cutoff at size $s$ and target $B$ is $s/B$. This choice
allows the removed part to be bounded in $\ell_1$, while ensuring the
pointwise gradient bound required for the retained part.

\begin{lemma}[Peeling]\label{lem:peeling}
Let $s,B>0$, and let $v$ be supported on $V$, with $\norm v_2^2\le s$.
Put
\[
 T=\{j\in V:|v(j)|>s/B\},\qquad
 t=v\one_T,\qquad c=v\one_{V\setminus T}.
\]
Then $v=t+c$ and
\begin{equation}\label{eq:peeling-bounds}
 \norm t_1\le B,\qquad \norm c_\infty\le s/B,
 \qquad \norm c_2^2\le s.
\end{equation}
If the peel is performed at $u\in[-1,1]^q$, then
$|\ip{t}{y-u}|\le2B$ for every $y\in[-1,1]^q$.
\end{lemma}
\begin{proof}
For $j\in T$, $|v(j)|\le(B/s)v(j)^2$. Summing gives
$\norm t_1\le(B/s)\norm v_2^2\le B$. The retained bounds follow
from restriction and the definition of $T$. The last assertion uses
$\norm{y-u}_\infty\le2$.
\end{proof}

Starting with a residual $v$ of size at most $s_p$, where
$p\in\{0,\ldots,K+1\}$, normalize it as follows. At class $k\le K$, skip to $k+1$
if its size is at most $s_{k+1}$. Otherwise peel at the cutoff
$s_k/B_k$. If the retained vector has size greater than $s_{k+1}$,
keep it as the active restriction of class $k$. If its size is at most
$s_{k+1}$, continue with that vector at class $k+1$. Reaching $K+1$
returns a terminal residual.

\begin{lemma}[Finite normalization]\label{lem:normalization}
This procedure examines at most $K+1-p$ classes. The peeled vectors and
the returned residual sum to the input. Their nonzero coordinates retain
the original input coefficients. Peel labels increase strictly, and an
active output of class $k$ satisfies
\begin{equation}\label{eq:restriction-birth}
 s_k/2<\norm c_2^2\le s_k,\qquad
 \norm c_\infty\le s_k/B_k.
\end{equation}
A terminal output has squared norm at most $s_{K+1}$.
\end{lemma}
\begin{proof}
Induct on the number of available classes. A skip preserves the vector,
which already satisfies the next size bound. A peel gives the exact splitting and
bounds of \cref{lem:peeling}. Each continuation increases the class
index by one. Composing the resulting splittings proves the assertions.
\end{proof}

An active restriction is stored with its original row index, class,
coefficient vector $c$, and birth point $u$. These data remain fixed
during its lifetime. It stays active while $r_x(c)>s_k/2$ and retires
at the first subsequent state where $r_x(c)\le s_k/2$. At that state,
its surviving restriction is normalized beginning at class $k+1$.
A large row whose size has fallen to at most $M$ is normalized beginning
at class zero.

\begin{lemma}[Increasing class labels]\label{lem:one-birth}
For each original row, there is at most one birth and at most one peel
in each class. Every coefficient in each of its restrictions is either
zero or the corresponding coefficient of the original row. At any state,
a row has at most one active unsigned restriction.
\end{lemma}
\begin{proof}
A class-$k$ restriction retires with surviving size at most $s_{k+1}$.
Its next normalization begins at $k+1$, and all labels encountered in a
normalization increase strictly. Restriction and peeling preserve the
coefficient assertion. Each normalization returns at most one active
restriction, which replaces the one just retired.
\end{proof}

\subsection{Regularized discrepancy and dangerous restrictions}
\label{sec:regularization}
Let a class-$k$ restriction $c$ be born at $u$. For each
$\varepsilon\in\{-1,1\}$ set
\begin{align}
 Z_c^\varepsilon(x)
 &=\varepsilon\ip{c}{x-u}
    +\beta_k\bigl(\cE_x(c)-\cE_u(c)\bigr),
 &\beta_k&=\frac{B_k}{2s_k},\label{eq:regularized}\\
 e_c^\varepsilon(x)
 &=\varepsilon c-2\beta_k(c^{\odot2}\odot x).
 &&\label{eq:gradient}
\end{align}
The coefficient vector in the energy includes the entire birth support.
Frozen coordinates make no subsequent contribution to its increment.
Thus $Z_c^\varepsilon(x)=Y_c^\varepsilon(x)-Y_c^\varepsilon(u)$, where
$Y_c^\varepsilon(x)=\varepsilon\ip{c}{x}+\beta_k\cE_x(c)$ is the
Bansal--Jiang regularized discrepancy for the signed row
$\varepsilon c$. Restricted to the alive coordinates, $e_c^\varepsilon(x)$
is the corresponding row of $E$ in~\cite[Section~3.2.1]{BJ25}.
Centering at birth leaves this gradient unchanged.

Each sign is initially \emph{live}. If it first reaches $B_k$ on a step
where the restriction remains active, that signed record becomes
\emph{dangerous}. Its moving gradient, restricted to the alive set, is
then blocked until retirement. This is a first-passage convention: the
mark remains if the later value of $Z_c^\varepsilon$ falls below $B_k$.
Retirement is determined by the unsigned restriction and applies to both
signs at the same new state.

Activity is a property of the unsigned restriction; live and dangerous
are statuses of its signed records. After one sign becomes dangerous,
the restriction stays active until its size halves. Both signed gradients
continue to belong to the class's affine matrix. The dangerous gradients
also belong to the blocking subspace, with their values recomputed at
the current point.

The peeling bound gives $2\beta_k|c(j)|\le1$. Consequently, throughout
the cube,
\begin{equation}\label{eq:gradient-bound}
 |e_c^\varepsilon(x)(j)|\le2|c(j)|,
 \qquad \norm{e_c^\varepsilon(x)}_2\le2\sqrt{s_k}.
\end{equation}

\begin{lemma}[Regularized increment and drift]\label{lem:increment}
For every $v\in\R^q$ and real $t$,
\begin{equation}\label{eq:exact-increment}
 Z_c^\varepsilon(x+tv)-Z_c^\varepsilon(x)
 =t\ip{e_c^\varepsilon(x)}{v}
  -\beta_kt^2\sum_jc(j)^2v(j)^2.
\end{equation}
If a centered direction has covariance $Q\preceq\eta\diag Q$, then
\begin{equation}\label{eq:drift}
 \E\bigl[Z_c^\varepsilon(x+\gamma v)-Z_c^\varepsilon(x)\bigr]
 \le-\theta_k\gamma^2
       (e_c^\varepsilon(x))^{\mathsf T}Qe_c^\varepsilon(x),
 \qquad \theta_k=\frac{B_k}{8\eta s_k}.
\end{equation}
If the moving gradient is orthogonal to $v$, the increment in
\eqref{eq:exact-increment} is nonpositive pointwise.
\end{lemma}
\begin{proof}
Expand the squares in \eqref{eq:regularized}. For the expectation, put
$C=e^{\mathsf T}Qe$ and $D=\sum_jc(j)^2Q(j,j)$. Spectral
independence and \eqref{eq:gradient-bound} give $C\le4\eta D$.
The mean increment is $-\beta_k\gamma^2D$, and
$\beta_k/(4\eta)=\theta_k$. The final assertion follows directly
from \eqref{eq:exact-increment}.
\end{proof}

Regularization therefore provides the same sign of quadratic correction
before and after a threshold crossing. The covariance controls the
expected linear response while a record is live. Blocking its moving
gradient later controls that response for each individual direction.

\begin{lemma}[Discrepancy during one lifetime]\label{lem:lifetime}
Suppose all directions have Euclidean norm one, and every dangerous
moving gradient is blocked. At every point during the lifetime of a
class-$k$ restriction, including its retirement point $z$,
\begin{equation}\label{eq:lifetime}
 |\ip{c}{z-u}|\le\frac32 B_k+2\gamma\sqrt{s_k}.
\end{equation}
\end{lemma}
\begin{proof}
Before crossing, $Z_c^\varepsilon<B_k$. By
\eqref{eq:exact-increment} and \eqref{eq:gradient-bound}, a step can
increase it by at most $2\gamma\sqrt{s_k}$, since the quadratic term
is nonpositive. A surviving crossing activates the blocking constraint;
a simultaneous retirement ends that lifetime at the same endpoint.
After the mark is active, \cref{lem:increment} makes subsequent
increments nonpositive. Hence both signs satisfy
$Z_c^\varepsilon\le B_k+2\gamma\sqrt{s_k}$ through retirement.
The two energy values lie in $[0,\norm c_2^2]$, so their regularized
difference has absolute value at most $\beta_ks_k=B_k/2$.
Apply \eqref{eq:regularized} to both signs.
\end{proof}

\subsection{The row identity through retirement}\label{sec:row-identity}
For a current state $(x,V)$ define the affine section
\[
 \cF(x,V)=\{y\in[-1,1]^q:y(j)=x(j)\text{ for }j\notin V\}.
\]
Buffered frozen coordinates are fixed at their current values in this
section. Their rounding to signs will be a final, separately estimated
operation.

A peeled vector $t$ is stored with its anchor point $u$, because its
future contribution is $\ip{t}{y-u}$. A completed restriction $c$
is stored with its birth point $u$ and retirement point $z$; it contributes
the constant $\ip{c}{z-u}$. A terminal residual keeps its final anchor.

\begin{lemma}[Affine row identity]\label{lem:row-identity}
Assume that previously frozen coordinates never move, large rows are
blocked through their entry step, and restrictions are normalized as
above. For every original row that has entered normalization, and every
$y\in\cF(x,V)$,
\begin{equation}\label{eq:row-identity}
 \ip{a_i}{y-b}
 =\sum_{(t,u)\in\cP_i}\ip{t}{y-u}
  +\sum_{(c,u,z)\in\cC_i}\ip{c}{z-u}
  +\cR_i(y).
\end{equation}
Here $\cP_i$ and $\cC_i$ are its actual peeled and completed records.
For an active row, $\cR_i(y)=\ip{c_i}{y-u_i}$ for its current
restriction. For a terminal row it is the analogous expression for its
terminal residual. A still-large row has $\ip{a_i}{x-b}=0$.
\end{lemma}
\begin{proof}
At the point $u$ where a row first enters normalization, its accumulated
discrepancy is zero: it has been blocked at every previous step,
including the entry step. For $y\in\cF(u,V_u)$,
\[
 \ip{a_i}{y-b}=\ip{a_i}{y-u}
              =\ip{a_i\one_{V_u}}{y-u}.
\]
Normalization gives the initial identity by its exact vector splitting.

For each later step, permanent freezing gives
$\cF(x^+,V^+)\subseteq\cF(x,V)$. Thus every unchanged affine
expression restricts to the new section. If a restriction retires at
$x^+$, then on that section
\begin{equation}\label{eq:retirement-splitting}
 \ip{c}{y-u}
 =\ip{c}{x^+-u}+\ip{c\one_{V^+}}{y-x^+}.
\end{equation}
The first term is stored as a completed contribution. Normalization
splits the second term into the new anchored peels and residual. This
is exactly the update of \eqref{eq:row-identity}. A large row that
remains large has zero step increment, and a newly entering large row
is covered by the initialization argument. Induction proves the result.
\end{proof}

\paragraph{Example: the row identity after retirement.}
Suppose a row enters at $u$, where normalization gives a peeled vector
$t_0$ and an active restriction $c_0$. Suppose $c_0$ retires at $z$,
and its surviving restriction is normalized as
$c_0\one_{V_z}=t_1+c_1$, with one peel and one active output for
simplicity. On the new frozen-coordinate section, the row identity is
\begin{equation}\label{eq:worked-retirement}
 \ip{a_i}{y-b}
 =\ip{t_0}{y-u}+\ip{c_0}{z-u}
  +\ip{t_1}{y-z}+\ip{c_1}{y-z}.
\end{equation}
The coefficient vectors $c_0$ and $c_1$ may share coordinates. Their
terms describe the successive intervals from $u$ to $z$ and from $z$
to $y$. The coordinates removed from $c_0$ at retirement are frozen at
$z$, which is why they contribute only to the completed interval.
The original anchor $u$ of $t_0$ is retained. Multiple immediate peels
replace $t_1$ by their sum, with the common anchor $z$.

\begin{corollary}[Row estimate from the history]\label{cor:row-bound}
At every state produced by these updates and the blocking rule,
\begin{equation}\label{eq:row-bound}
 \norm{A(x-b)}_\infty
 \le\frac72\sum_{k=0}^{K}B_k+2+21\gamma.
\end{equation}
\end{corollary}
\begin{proof}
Evaluate \eqref{eq:row-identity} at $y=x$. Each peeled vector costs
at most $2B_k$ by \cref{lem:peeling}. Each completed or current
restriction costs at most $3B_k/2+2\gamma\sqrt{s_k}$ by
\cref{lem:lifetime}, evaluated at its own endpoint. A terminal
residual has $\ell_1$ norm at most $\sqrt{q s_{K+1}}\le1$, hence
costs at most two. By \cref{lem:one-birth}, at most one peel and one
tracked interval occur in each class. A still-large row has zero
accumulated discrepancy. Finally,
\[
 2\sum_{k\ge0}\sqrt{s_k}
 =\frac{2\sqrt{60/7}}{1-2^{-1/2}}<21.
\]
For an exact comparison, use $\sqrt{60/7}<3$ and
$2^{-1/2}<5/7$.
\end{proof}

\section{Historical exponential sums and complete updates}\label{sec:capital}
The covariance budget concerns the current restrictions, while the
potential must retain information from restrictions that have already
retired. We define the historical quantities so that each status change
has an exact effect on the potential.

Fix a class $k$. By \cref{lem:one-birth}, the signed births in that
class can be indexed by slots $(i,\varepsilon)$, with $i\in[m]$ and
$\varepsilon\in\{-1,1\}$. A slot acquires its coefficient vector and
birth point only at its unique birth. All later quantities for that
slot use those same coefficients. Write $a$ for a signed slot and
$c_a$ for its unsigned coefficient vector.

\subsection{Exponential sums and the two column estimates}
The original decoupling proof~\cite[Section~3.2.5]{BJ25} controls
$W_t=\sum_i\min\{\Phi_t(i),e^{\lambda B}\}$, where
$\Phi_t(i)=e^{\lambda Z_t(i)}$. Here $X$ denotes each capped factor,
and $\Phi$ is the column-weighted sum over restriction histories.

Fix $L_k>0$ and put $\lambda_k=L_k/B_k$. A record starts with $X=1$ at
birth. While it is live, multiply $X$ by the exponential of its
regularized increment and cap the result at $e^{L_k}$. A dangerous or retired
record retains its value of $X$. More explicitly, an old live record has
\begin{equation}\label{eq:capital-update}
 X_a^+=\min\{e^{L_k},X_a e^{\lambda_k\Delta Z_a}\}.
\end{equation}
This last update is performed even if the restriction retires at the
new point. At a live state, $X_a=e^{\lambda_k Z_a}$ and $Z_a<B_k$.
At a surviving first crossing, \eqref{eq:capital-update} sets
$X_a^+=e^{L_k}$ exactly.

For a fixed column $j$ and every born slot in class $k$, put
\[
 q_a^{(j)}=\frac{c_a(j)^2}{s_k}.
\]
This normalization gives an active restriction total weight
$\sum_{j\in V}q_a^{(j)}>1/2$. Summing the weighted exponentials over the alive
columns will therefore count dangerous restrictions. At a fixed column,
the column norm hypothesis bounds all births in this class, because
each original row contributes at most once per sign.
Unallocated slots contribute zero. Define
\begin{equation}\label{eq:capital-definitions}
 I^{(j,k)}=\sum_{a\text{ born}}q_a^{(j)},\qquad
 \Phi^{(j,k)}=\sum_{a\text{ born}}q_a^{(j)}X_a,
 \qquad G^{(j,k)}=\Phi^{(j,k)}-I^{(j,k)}.
\end{equation}
The quantity $I^{(j,k)}$ is the total mass inserted at births. A new record
contributes $q_a^{(j)}$ to both $I$ and $\Phi$,
so its initial contribution to $G$ is zero. The difference $G$ may be
negative. The dangerous-family count uses the nonnegative sum
$\Phi=I+G$, while the outer exponential controls $G$ from above.

\begin{lemma}[Historical column mass]\label{lem:immigration}
At every state,
\begin{equation}\label{eq:immigration}
 I^{(j,k)}\le\frac2{s_k},\qquad
 0\le\Phi^{(j,k)}\le\frac{2e^{L_k}}{s_k}.
\end{equation}
If $\cD_k$ is the current dangerous signed family, then
\begin{equation}\label{eq:danger-capital}
 e^{L_k}\sum_{a\in\cD_k}q_a^{(j)}
 \le I^{(j,k)}+G^{(j,k)}.
\end{equation}
\end{lemma}
\begin{proof}
For each original row there is at most one birth in class $k$, and its
coefficient at $j$ is either zero or $a_i(j)$. Including both signs,
\[
 \sum_{a\text{ born in }k}c_a(j)^2
 \le2\sum_i a_i(j)^2\le2.
\]
This proves the birth-mass bound. Every born record has $X$ in
$(0,e^{L_k}]$, which gives the bound on $\Phi$. Each current dangerous
record retains the value $e^{L_k}$ reached at its first crossing.
Its exponential is among the nonnegative terms in $\Phi=I+G$, proving
\eqref{eq:danger-capital}.
\end{proof}

There is a separate estimate for the current unsigned family. If
$\cE_k^{\mathrm{cur}}$ is the total size on alive coordinates of its
class-$k$ restrictions, then
\begin{equation}\label{eq:current-mass}
 \sum_k\cE_k^{\mathrm{cur}}
 =\sum_{j\in V}\sum_{i\text{ current}}c_i(j)^2
 \le\sum_{j\in V}\sum_i a_i(j)^2\le h.
\end{equation}
Here each original row has at most one current restriction, across all
classes. In \eqref{eq:immigration}, the class is fixed and every
historical birth in that class is included. The historical estimate
controls the exponential sum; the current estimate controls the affine SDP
charge.

\subsection{The order of a complete event}\label{sec:event}
At a current state, the blocking subspace contains the alive restrictions
of large rows, all dangerous signed moving gradients, and the current
point $x|_V$. For each class, the affine matrix $E_k$ contains both
signed gradients of every current active restriction, restricted to $V$.
A feasible covariance determines the finite direction law of
\cref{lem:direction-law}.

For each atom, the complete successor is formed in the following order.
First set $x^+=x+\gamma v$ using the old constraints. Update every
old live signed exponential by \eqref{eq:capital-update}, using its old
coefficient vector and birth point. Old dangerous exponentials stay fixed,
while their regularized discrepancies are evaluated at the new point.
Next determine $V^+$ and the restrictions whose alive size has fallen
to their retirement boundary. Archive those restrictions at $x^+$ with
their final signed discrepancies and exponentials. Among the surviving
restrictions, mark every newly crossing sign dangerous. Normalize the surviving
restrictions of retired rows, and any large rows that have become small,
and insert the resulting births at $x^+$ with $X=1$. The new SDP
data are formed only after these operations.

In particular, a threshold crossing simultaneous with retirement creates
no new dangerous constraint. If $\cT$ denotes the old signed slots whose
restrictions retire, and $\cC$ denotes the newly crossing old live
slots, the new dangerous family is
\begin{equation}\label{eq:danger-update}
 \cD^+=(\cD\cup\cC)\setminus\cT.
\end{equation}
Both signs use the same unsigned retirement predicate. Previously
retired records remain unchanged, and all historical coefficients and
birth points remain fixed.

\begin{lemma}[Exponential-sum increment for the complete event]
\label{lem:event-identity}
Let $\mathcal U$ be the old live signed family in a fixed class, and
write $r_a=q_a^{(j)}X_a$. For every provisional complete successor,
\begin{align}
 G^{+(j,k)}-G^{(j,k)}
 &=\sum_{a\in\mathcal U}q_a^{(j)}
       \bigl(\min\{e^{L_k},X_a e^{\lambda_k\Delta Z_a}\}-X_a\bigr)
       \label{eq:event-equality}\\
 &\le\sum_{a\in\mathcal U}r_a
             \bigl(e^{\lambda_k\Delta Z_a}-1\bigr).
       \label{eq:event-upper}
\end{align}
\end{lemma}
\begin{proof}
Every birth adds the same mass to $\Phi$ and $I$. All old
records outside $\mathcal U$ retain their exponential values, including previously
retired and dangerous records. The old live family receives the update
in \eqref{eq:capital-update} before retirement is classified. These
facts give \eqref{eq:event-equality}. Replacing the capped value by its
uncapped value gives \eqref{eq:event-upper}.
\end{proof}

The right side contains only predecessor data and the regularized
increments along $v$. New coefficients may depend on $v$ through
freezing and normalization, but their contribution to $G$ is exactly
zero at birth. We can therefore differentiate the smooth expression in
the next section while retaining a pointwise bound on the complete update.

\subsection{The outer exponential and its normalization}
Assume the parameters satisfy
\begin{equation}\label{eq:lambda-theta}
 0<\lambda_k\le\theta_k=\frac{B_k}{8\eta s_k}.
\end{equation}
Set
\begin{equation}\label{eq:potential-normalization}
 \rho_k=\frac{1973}{100B_k^2},\qquad
 K_k=\frac{\lambda_k\rho_k e^{L_k}}
             {2(\theta_k-\lambda_k/2)}.
\end{equation}
All these quantities are positive. The coefficient cap gives
\begin{equation}\label{eq:mass-cap}
 \alpha_k q_a^{(j)}\le\rho_k,
 \qquad\text{and hence}\qquad
 \alpha_k r_a^2\le\rho_k e^{L_k}r_a.
\end{equation}

The choice of $K_k$ follows the covariance--drift comparison in the
Bansal--Jiang decoupling argument~\cite[Section~3.2.5]{BJ25}, applied
here to column-weighted historical exponential sums. If $C$ is the covariance
matrix of the old live signed responses, then
\[
 r^{\mathsf T}Cr
 \le\alpha_k\sum_a r_a^2C_{aa}
 \le\rho_k e^{L_k}\sum_a r_aC_{aa}.
\]
The affine constraint gives the first inequality, and the coefficient
and exponential caps give the second. Thus the variance of the weighted sum
is controlled by the same diagonal quantity as its negative drift.

\Needspace{11\baselineskip}
The resulting upper bound for the second derivative of the normalized
outer exponential, computed in \cref{sec:cancellation}, is
\[
 \left[-\frac{2\lambda_k(\theta_k-\lambda_k/2)}{K_k}
       +\frac{\lambda_k^2\rho_k e^{L_k}}{K_k^2}\right]
       \sum_a r_aC_{aa}.
\]
The first term combines regularized drift with the second derivative
of each inner exponential. The second comes from the covariance of
their sum. Equation~\eqref{eq:potential-normalization} chooses $K_k$
so that the bracket vanishes. The next section keeps the slack in this
comparison and uses the buffer to obtain a finite-step estimate.

\section{The finite-step exponential estimate}\label{sec:local}
Fix a predecessor, a column $j$, and a class $k$. All records and
weights in this section are those of this one predecessor. Abbreviate
\[
 \begin{gathered}
 s=s_k,\quad B=B_k,\quad L=L_k,\quad T=e^L,\quad
 \lambda=\lambda_k,\\
 \beta=\beta_k,\quad \theta=\theta_k,\quad\alpha=\alpha_k,\quad
 \rho=\rho_k,\quad K'=K_k.
 \end{gathered}
\]
Let $a$ range over the old live signed family, and put
$r_a=q_a^{(j)}X_a$. The covariance is $Q=U/\Tr U$. Define
\begin{align}
 u_a(v)&=\ip{e_a}{v},&
 d_a(v)&=\sum_pc_a(p)^2v(p)^2,\label{eq:local-responses}\\
 C_{ab}&=e_a^{\mathsf T}Qe_b,&
 D_a&=\sum_pc_a(p)^2Q(p,p),&
 S&=\sum_a r_aC_{aa}.\label{eq:local-covariances}
\end{align}
By \cref{lem:immigration} and the geometric bounds,
\begin{equation}\label{eq:local-uniform-bounds}
 \sum_a r_a\le\frac{2T}{s},\qquad
 |u_a(v)|\le2\sqrt s,\qquad
 0\le d_a(v)\le\frac{s^2}{B^2}
\end{equation}
for every unit direction.

The smooth upper increment for a real auxiliary parameter $t$ is
\begin{equation}\label{eq:smooth-increment}
 \cH(t,v)=\sum_a r_a
       \left(e^{\lambda t u_a(v)-\lambda\beta t^2d_a(v)}-1\right),
 \qquad F(t)=\E e^{\cH(t,v)/K'}.
\end{equation}
The identity for the weighted exponential sum and the exact regularized increment give, for the
same complete successor,
\begin{equation}\label{eq:actual-smooth-comparison}
 \Delta G^{(j,k)}\le\cH(\gamma,v).
\end{equation}
We differentiate \eqref{eq:smooth-increment} with the predecessor and
its law fixed. The expectation is a finite sum.

\subsection{The covariance cancellation}\label{sec:cancellation}
Suppose the coordinate and signed affine constraints hold:
\begin{equation}\label{eq:local-SI}
 Q\preceq\eta\diag Q,
 \qquad C\preceq\alpha\diag C.
\end{equation}
The second inequality is inherited by the old live family from the
principal submatrix of the class's full signed affine block.
Set
\begin{equation}\label{eq:slacks}
\begin{aligned}
 \Delta_{\mathrm{coord}}&=\sum_a r_a(4\eta D_a-C_{aa}),\\
 \Delta_{\mathrm{aff}}&=\alpha\sum_a r_a^2C_{aa}-r^{\mathsf T}Cr,\\
 \Delta_{\mathrm{cap}}&=\rho T S-\alpha\sum_a r_a^2C_{aa}.
\end{aligned}
\end{equation}

\begin{lemma}[Exact second derivative]\label{lem:second-derivative}
All three quantities in \eqref{eq:slacks} are nonnegative, and
\begin{equation}\label{eq:second-identity}
 (K')^2F''(0)
 =-2\lambda\theta K'\Delta_{\mathrm{coord}}
  -\lambda^2\Delta_{\mathrm{aff}}
  -\lambda^2\Delta_{\mathrm{cap}}.
\end{equation}
\end{lemma}
\begin{proof}
For each signed record, coordinate spectral independence gives
\begin{align*}
 4\eta D_a-C_{aa}
 &=\eta\sum_p\bigl(4c_a(p)^2-e_a(p)^2\bigr)Q(p,p)\\
 &\quad+e_a^{\mathsf T}(\eta\diag Q-Q)e_a\ge0.
\end{align*}
The gradient bound proves the sign of the first term. Affine spectral
independence gives $\Delta_{\mathrm{aff}}\ge0$, and
\eqref{eq:mass-cap} gives $\Delta_{\mathrm{cap}}\ge0$.

Differentiation of the full outer exponential gives
\begin{equation}\label{eq:second-direct}
 (K')^2F''(0)
 =-2\lambda\beta K'\sum_a r_aD_a
    +\lambda^2K'S+\lambda^2r^{\mathsf T}Cr.
\end{equation}
Use $\beta=4\eta\theta$ and
$r^{\mathsf T}Cr=\rho TS-\Delta_{\mathrm{aff}}-\Delta_{\mathrm{cap}}$.
The remaining coefficient of $S$ is
\[
 -2\lambda\theta K'+\lambda^2K'+\lambda^2\rho T=0
\]
by \eqref{eq:potential-normalization}. This proves the identity.
\end{proof}

The outer covariance term in \eqref{eq:second-direct} is essential:
it is the variance of a sum of signed responses. The affine constraint
bounds that sum with the live exponential factors still attached. The cap
then converts the squared weights back to the weights appearing in the
drift term.

\subsection{Negative curvature from the buffer}\label{sec:buffer}
Let the covariance be supported on coordinates where
$|x(p)|\le1-\delta$, with $0<\delta\le1$. On this support the
coefficient cap gives
\begin{equation}\label{eq:buffer-gradient}
 |e_a(p)|\le(2-\delta)|c_a(p)|,
 \qquad C_{aa}\le\eta(2-\delta)^2D_a.
\end{equation}
Indeed $2\beta|c_a(p)|\le1$, so the first inequality follows from
the explicit gradient. The second follows from coordinate spectral
independence. Define the normalized weighted energy
\begin{equation}\label{eq:weighted-energy}
 \mathcal D=\frac1{K'}\sum_a r_aD_a\ge0.
\end{equation}
It follows that
$\Delta_{\mathrm{coord}}\ge\eta(4\delta-\delta^2)K'\mathcal D$.
Substituting this into \eqref{eq:second-identity} gives
\begin{equation}\label{eq:negative-curvature}
 F''(0)\le-\frac{\lambda\beta}{2}
                    (4\delta-\delta^2)\mathcal D.
\end{equation}
In the actual walk, the covariance is supported on the buffered alive
set, so we may take $\delta=\gamma$.

After multiplication by the Taylor factor $\gamma^2/2$, this buffer
contributes a negative term of order $\gamma^3\mathcal D$. Symmetry
will remove the odd Taylor terms. We bound the remaining error by a
constant times $\gamma^4\mathcal D$, so it can be absorbed by a mesh
bound depending only on the class parameters. The common factor
$\mathcal D$ is essential for this uniformity, since the energy can
approach zero at later states. The following lemma provides exactly
that weighted remainder estimate.

\begin{lemma}[Weighted fourth derivative]\label{lem:weighted-fourth}
Let $\omega_a\ge0$, $\sum_a\omega_a\le\sigma$, $|p_a|\le P_0$,
and $0\le q_a\le Q_0$, where $\sigma,P_0,Q_0$ are nonnegative.
Put
\[
 J(t)=\sum_a\omega_a(e^{p_at-q_at^2}-1),\qquad
 W=\sum_a\omega_a(p_a^2+q_a).
\]
There are explicit $h_b>0$ and $C_4\ge0$, depending only on
$\sigma,P_0,Q_0$, such that
\begin{equation}\label{eq:weighted-fourth}
 \left|\frac{d^4}{dt^4}e^{J(t)}\right|\le C_4 W
 \qquad(|t|\le h_b).
\end{equation}
The constants and a complete differentiation proof are given in
\cref{app:derivatives}.
\end{lemma}

For each supported unit direction, apply the lemma with
\begin{equation}\label{eq:fourth-substitution}
 \omega_a=r_a/K',\qquad
 p_a(v)=\lambda u_a(v),\qquad
 q_a(v)=\lambda\beta d_a(v).
\end{equation}
Uniform parameter choices, depending only on the class schedule, are
\begin{equation}\label{eq:uniform-fourth-parameters}
 \sigma=\frac{2T}{sK'},\qquad
 P_0=\lambda(1+s),\qquad Q_0=\frac{\lambda s}{2B}.
\end{equation}
Here $2\sqrt s\le1+s$ bounds $p_a$, and the coefficient cap in
\eqref{eq:local-uniform-bounds} bounds $q_a$. Moreover,
\begin{align}
 \E W(v)
 &=\frac1{K'}\sum_a r_a
           (\lambda^2 C_{aa}+\lambda\beta D_a)\notag\\
 &\le(4\eta\lambda^2+\lambda\beta)\mathcal D.
 \label{eq:expected-weight}
\end{align}
Set
\begin{equation}\label{eq:analytic-mesh}
 \Gamma_4=C_4(4\eta\lambda^2+\lambda\beta),
 \qquad
 h_* =\min\left\{\frac12,h_b,\frac{6\lambda\beta}{\Gamma_4}\right\}.
\end{equation}
For the positive class parameters in \eqref{eq:uniform-fourth-parameters},
the constants in \cref{app:derivatives} give $\Gamma_4>0$, so
$h_*>0$. Taking the finite average in \eqref{eq:weighted-fourth} gives
\begin{equation}\label{eq:averaged-fourth}
 |F^{(4)}(t)|\le\Gamma_4\mathcal D\qquad(|t|\le h_b).
\end{equation}
The bounds use historical mass and the coefficient cap; they do not
require a positive lower bound on an eigenvalue or on a direction's
probability.

\Needspace{27\baselineskip}
\subsection{The one-step estimate for actual successors}
\begin{theorem}[One-step exponential bound]\label{thm:local-estimate}
Fix a class $k$ and a column $j$, with the positive parameters in
\eqref{eq:fixed-parameters}, \eqref{eq:regularized},
\eqref{eq:lambda-theta}, and \eqref{eq:potential-normalization}, and
$\lambda_k=L_k/B_k$. Let $x\in[-1,1]^q$ be a predecessor whose old
live signed records satisfy
\[
 \norm{c_a}_2^2\le s_k,\qquad
 \norm{c_a}_\infty\le s_k/B_k,\qquad
 0<X_a=e^{\lambda_k Z_a}<e^{L_k},\qquad
 \sum_a q_a^{(j)}\le2/s_k,
\]
where $q_a^{(j)}=c_a(j)^2/s_k$ and the signed gradients are given by
\eqref{eq:gradient}. Use a finite centrally symmetric law on unit
directions supported on $\{p:|x(p)|<1-\gamma\}$, with covariance $Q$.
Assume that $Q$ and the old live signed response covariance $C$ satisfy
\[
 Q\preceq\eta\diag Q,\qquad C\preceq\alpha_k\diag C,
\]
and form each complete successor as in \cref{sec:event}.
If $0<\gamma\le h_*$, with $h_*$ defined in
\eqref{eq:analytic-mesh}, then
\begin{equation}\label{eq:one-step}
 \E\exp\left(\frac{\Delta G^{(j,k)}}{K_k}\right)
 \le1-\frac12\lambda_k\beta_k\gamma^3\mathcal D\le1.
\end{equation}
Here the normalized weighted energy is
\[
 \mathcal D=\frac1{K_k}\sum_a q_a^{(j)}X_a
                  \sum_p c_a(p)^2Q(p,p).
\]
For fixed class parameters, the allowed mesh is uniform over all
predecessors and direction laws satisfying these hypotheses.
\end{theorem}
\begin{proof}
With the predecessor fixed, symmetry gives
$\cH(-t,v)=\cH(t,-v)$, so $F(-t)=F(t)$. Therefore
$F'(0)=F'''(0)=0$. Taylor's formula with integral remainder and
\eqref{eq:averaged-fourth} yield
\[
 F(\gamma)\le1+\frac{\gamma^2}{2}F''(0)
                     +\frac{\Gamma_4\gamma^4}{24}\mathcal D.
\]
Apply \eqref{eq:negative-curvature} with $\delta=\gamma$:
\begin{align*}
 F(\gamma)
 &\le1-\lambda\beta(1-\gamma/4)\gamma^3\mathcal D
                +\frac{\Gamma_4\gamma^4}{24}\mathcal D\\
 &\le1-\frac12\lambda\beta\gamma^3\mathcal D.
\end{align*}
The last line uses $\gamma\le1$ and
$\Gamma_4\gamma\le6\lambda\beta$.
Finally \eqref{eq:actual-smooth-comparison} holds pointwise for the
same complete successors. Monotonicity of the exponential transfers the
bound to the actual event.

If $\mathcal D=0$, then \eqref{eq:expected-weight} gives
$\E W(v)=0$. Every positive-weight atom has $W(v)=0$, and each
inner term with positive weight has $p_a=q_a=0$. Thus the smooth
exponential equals one, which also verifies this case directly.
\end{proof}

Symmetry is used for the smooth predecessor expression. The freeze and
retirement predicates in the actual event may have different outcomes
on $v$ and $-v$; the pointwise comparison already includes both outcomes.
All mesh constants were obtained from the fixed class parameters before
choosing any future covariance.

\section{The schedule and the covariance budget}\label{sec:schedule}
We now choose the class parameters at the prescribed depth $R\ge0$.
The choices must bound dangerous restrictions by a fixed fraction of the
current alive dimension and make the row thresholds summable.

For every $k\ge0$, set
\begin{align}
 w_0&=\frac{17}{140},\qquad
 w_k=\frac{123}{1190}\left(\frac{15}{17}\right)^{k-1}
                       \quad(k\ge1),\qquad
 \delta_k=\frac{43}{80}w_k,\label{eq:weights}\\
 L_k&=\log\frac{32}{\delta_ks_k},\qquad
 \lambda_k=\frac{L_k}{B_k(R)},\notag\\
 B_k(R)&=\left(\frac{1925648}{675}
              \frac{s_kL_k(L_k+R)}{\delta_k}\right)^{1/4}
             +\sqrt{\frac{244}{9}s_kL_k}.
 \label{eq:thresholds}
\end{align}
Use these thresholds in the normalization of \cref{sec:restrictions},
and use \eqref{eq:potential-normalization} for $\rho_k,K_k$.
For the finite cutoff, define the potential of the complete state by
\begin{equation}\label{eq:potential}
 \Psi=\sum_{k=0}^{K}\sum_{j=1}^{q}
                   w_k e^{G^{(j,k)}/K_k}.
\end{equation}
The historical records determine $G^{(j,k)}$. At initialization every
born record has $X=1$, so $G^{(j,k)}=0$ and
$\Psi_0=q\sum_{k=0}^{K}w_k\le q$.
The potential sums over all $q$ initially fractional columns throughout
the construction. When a column freezes, its historical accounts remain
in this sum and continue to follow their records. The current alive set
enters the counting argument only when we apply Jensen's inequality.

\subsection{Why the threshold has a fourth root}
The dangerous-family count proved below has a term
\[
 2h e^{-L_k}\left[\frac2{s_k}
             +K_k\left(\log\frac qh+\log\frac1{w_k}\right)\right].
\]
The choice of $L_k$ makes its first contribution exactly
$\delta_kh/8$. For the second contribution,
\[
 e^{-L_k}K_k
 =\frac{\lambda_k\rho_k}{2(\theta_k-\lambda_k/2)}
 \le\frac{\lambda_k\rho_k}{\theta_k}
 \asymp\frac{s_kL_k}{B_k^4}.
\]
Thus $B_k^4$ must dominate $s_kL_k(L_k+R)/\delta_k$. The
square-root summand in \eqref{eq:thresholds} guarantees
$\lambda_k\le\theta_k$ and leaves strict slack in the fourth-power
inequality. The particular rational constants make the total SDP budget
and the final threshold sum explicit.

\begin{lemma}[Schedule inequalities]\label{lem:schedule}
For every $R\ge0$, the weights have total mass
\begin{equation}\label{eq:weight-sums}
 \sum_{k\ge0}w_k=1,\qquad
 \sum_{k\ge0}\delta_k=\frac{43}{80}.
\end{equation}
For each class $k\ge0$, the logarithmic and drift parameters satisfy
\begin{equation}\label{eq:schedule-basic}
 L_k>1,\qquad \log(1/w_k)<L_k,\qquad
 0<\lambda_k\le\theta_k.
\end{equation}
The two coefficients in the dangerous-restriction count obey
\begin{equation}\label{eq:K-bound}
\begin{aligned}
 \frac{4e^{-L_k}}{s_k}&=\frac{\delta_k}{8},\\[3pt]
 e^{-L_k}K_k&\le\frac{120353}{225}\,\frac{s_kL_k}{B_k^4}.
\end{aligned}
\end{equation}
The thresholds also satisfy the strict bound
\begin{equation}\label{eq:fourth-budget}
 \frac{240706}{225}\,\frac{s_kL_k(L_k+R)}{B_k^4}
       <\frac{3\delta_k}{8}.
\end{equation}
\end{lemma}
\begin{proof}
The geometric series gives
$17/140+(123/1190)/(1-15/17)=1$. Also
\[
 L_k-\log(1/w_k)=\log\frac{32}{(43/80)s_k}>0.
\]
\begin{samepage}
The first two logarithmic parameters satisfy
\[
 e^{L_0}=\frac{125440}{2193}>3>e,
 \qquad e^{L_1}=\frac{2132480}{15867}>e^{L_0}.
\]
Together with $L_{k+1}-L_k=\log(34/15)>0$ for $k\ge1$, this gives $L_k>1$.
\end{samepage}
The square-root summand of $B_k$ gives
$B_k^2\ge(244/9)s_kL_k=8\eta s_kL_k$, which proves
$\lambda_k\le\theta_k$.

Since $e^{-L_k}=\delta_ks_k/32$, the first identity in
\eqref{eq:K-bound} is immediate. The second follows from
\[
 e^{-L_k}K_k
 \le\frac{\lambda_k\rho_k}{\theta_k}
 =8\eta\frac{1973}{100}\frac{s_kL_k}{B_k^4}
 =\frac{120353s_kL_k}{225B_k^4}.
\]
Finally $B_k$ is strictly larger than its positive fourth-root summand,
and
$\frac{240706}{225}/\frac{1925648}{675}=3/8$.
This proves \eqref{eq:fourth-budget}.
\end{proof}

The same formulas show why the thresholds are summable. For $k\ge1$,
\[
 \frac{s_k}{\delta_k}=\frac{s_1}{\delta_1}
                 \left(\frac{17}{30}\right)^{k-1},\qquad
 L_k=L_1+(k-1)\log\frac{34}{15}.
\]
Since $L_k>1$, we have $L_k(L_k+R)\le L_k^2(1+R)$ for $R\ge0$.
Thus \eqref{eq:thresholds} gives an absolute constant $C_0$ such that,
for every $k\ge1$ and $R\ge0$,
\[
 B_k(R)\le C_0(1+R)^{1/4}\sqrt{k+1}
 \left[\left(\frac{17}{30}\right)^{(k-1)/4}
                    +2^{-(k-1)/2}\right].
\]
Both series converge. The $k=0$ term has the same depth bound, so
$\sum_{k\ge0}B_k(R)=O((1+R)^{1/4})$, with an absolute constant
independent of the finite cutoff. Summing over the size classes therefore
introduces no further logarithmic factor. \Cref{prop:threshold-sum}
gives the explicit constant used in the endpoint bound.

\subsection{Counting dangerous restrictions at the current state}
\begin{lemma}\label{lem:danger-count}
Suppose the geometric records and exponential sums at a state satisfy the
properties above and $\Psi\le q$. If its alive count satisfies
$h>N_R(q)$, then, in every class,
\begin{equation}\label{eq:danger-budget}
 |\cD_k|<\frac{\delta_k}{2}h.
\end{equation}
\end{lemma}
\begin{proof}
For a fixed class, the positive terms of \eqref{eq:potential} give
$\sum_{j\in V}e^{G^{(j,k)}/K_k}\le q/w_k$. Jensen's inequality
therefore yields
\begin{equation}\label{eq:jensen}
 \frac1h\sum_{j\in V}\frac{G^{(j,k)}}{K_k}
 \le\log\frac{q}{w_kh}.
\end{equation}
Every dangerous signed restriction is active, so its alive normalized
mass is greater than $1/2$. In particular, including the case of an
empty family,
\[
 |\cD_k|\le2\sum_{a\in\cD_k}\sum_{j\in V}q_a^{(j)}.
\]
Sum \eqref{eq:danger-capital} over $V$, use
\eqref{eq:immigration} and \eqref{eq:jensen}, and obtain
\begin{equation}\label{eq:danger-general}
 |\cD_k|\le2he^{-L_k}
    \left[\frac2{s_k}
       +K_k\left(\log\frac qh+\log\frac1{w_k}\right)\right].
\end{equation}
Since $h$ is an integer and $h>N_R(q)$, we have $h\ge9$ and
$h\ge\lceil qe^{-R}\rceil\ge qe^{-R}$. Consequently
$\log(q/h)\le R$. Apply \cref{lem:schedule} to bound
\eqref{eq:danger-general} strictly by
$h(\delta_k/8+3\delta_k/8)=\delta_kh/2$.
\end{proof}

The Jensen estimate is taken over the current alive columns.
Its right side is determined by $q/h$ and the class weight. Historical
records entered only through the column estimate
\eqref{eq:immigration}. Thus the argument controls the current dangerous
family without a union bound over the row set or over possible histories.

\subsection{Feasibility of the next covariance}
\begin{proposition}\label{prop:feasibility}
At a state as in \cref{lem:danger-count}, there is a nonzero
$U\succeq0$, supported on $V\times V$, that annihilates the alive
restrictions of all large rows, all dangerous signed moving gradients,
and $x|_V$, and satisfies
\begin{gather}
 U(j,j)\le1,\qquad \Tr U\ge h/65536,
 \qquad U\preceq\frac{61}{18}\diag U,\label{eq:chosen-coordinate}\\
 E_kUE_k^{\mathsf T}\preceq\frac{1973}{100s_k}
                 \diag(E_kUE_k^{\mathsf T})
 \qquad(0\le k\le K).\label{eq:chosen-affine}
\end{gather}
Here $E_k$ contains the current signed moving gradients in class $k$.
\end{proposition}
\begin{proof}
Work in $\R^V$. Large rows cost at most $h/M$ dimensions.
The current point costs at most one. Since $h\ge9>M$, their combined
cost is strictly less than $2h/M=7h/30$. By
\cref{lem:danger-count}, the dangerous gradients cost at most
$h\sum_k\delta_k/2\le43h/160$. Thus the blocking subspace has
dimension less than $241h/480$.

Let $m_k$ be the number of signed rows in $E_k$. Each active unsigned
restriction has size greater than $s_k/2$, so
\[
 m_k\le\frac{4\cE_k^{\mathrm{cur}}}{s_k}.
\]
This inequality is valid for an empty class as well. By
\eqref{eq:current-mass},
\begin{equation}\label{eq:affine-charge}
 \sum_{k=0}^{K}\frac{m_k}{h\alpha_k}
 \le\frac{400}{1973h}\sum_k\cE_k^{\mathrm{cur}}
 \le\frac{400}{1973}.
\end{equation}
The remaining budget is strictly positive:
\begin{equation}\label{eq:exact-sdp-slack}
 \frac{239}{480}
 -\left(\frac{18}{61}+\frac1{65536}+\frac{400}{1973}\right)
 =\frac{9759761}{118311813120}>0.
\end{equation}
Apply \cref{cor:finite-row} with $\Delta=241/480$ and extend the
resulting matrix by zero outside $V$. Its trace bound makes it nonzero.
\end{proof}

Every hypothesis used to obtain $U$ concerns the predecessor state.
The next section selects a complete successor whose potential bound
ensures the same estimates at the next state.

\section{The finite trajectory and the endpoint}\label{sec:trajectory}
Fix $q>8$ and a prescribed depth $R\ge0$. Choose the finite cutoff and
class schedule as above. For each class let $h_{*,k}>0$ be the uniform
analytic bound in \eqref{eq:analytic-mesh}, and put
\begin{equation}\label{eq:common-mesh-bound}
 h_{\mathrm{an}}=\min_{0\le k\le K}h_{*,k}>0,
 \qquad d_b=\min_{1\le j\le q}(1-|b(j)|)>0.
\end{equation}
Both quantities are fixed before the trajectory begins.

\subsection{Preserving the complete state}
The state consists of the point, the current row restrictions, and the
historical records described in \cref{sec:restrictions,sec:capital}.
The inductive conditions are as follows. The point lies in the cube and
previously frozen coordinates have stayed fixed. Each original row is
large, active, or terminal according to the normalization rule. Its
coefficient and anchor records satisfy the affine row identity, and its
class labels increase. A current signed record is live with
$X=e^{\lambda_k Z}<e^{L_k}$, or dangerous with $X=e^{L_k}$.
Retired records retain their final value of $X$. All recorded
lifetimes satisfy \eqref{eq:lifetime}. Finally,
\begin{equation}\label{eq:inductive-potential}
 \Psi\le q.
\end{equation}
The birth-mass estimates and the current mass bound follow from these
records and the column norm hypothesis.

\begin{lemma}[Initialization]\label{lem:initialization}
For $0<\gamma<\min\{1/2,d_b\}$, normalization at $b$ produces a
state satisfying all these conditions.
\end{lemma}
\begin{proof}
All $q$ coordinates are alive. Leave the large rows blocked and normalize
every other row at $b$. The normalization identities initialize the row
records. Each birth has $Z=0$ and $X=1$, so every new signed record is
live and the dangerous families are empty. There are no completed
lifetimes. The initial exponential sum equals the birth mass, hence $G^{(j,k)}=0$.
The weight sum gives $\Psi_0=q\sum_{k=0}^{K}w_k\le q$.
\end{proof}

\begin{proposition}[A complete successor preserving the invariant]
\label{prop:successor}
Suppose a state satisfies the inductive conditions and has
$h>N_R(q)$ alive coordinates. If $0<\gamma\le h_{\mathrm{an}}$,
it has a complete successor satisfying the same conditions and
\begin{equation}\label{eq:selected-step}
 \Psi^+\le\Psi,\qquad \norm{x^+}_2^2=\norm x_2^2+\gamma^2.
\end{equation}
\end{proposition}
\begin{proof}
By \cref{prop:feasibility}, the predecessor has a feasible covariance.
Use its finite centrally symmetric law and form the complete successor
of \cref{sec:event} for each atom. Cube preservation and permanent
freezing hold for every atom by \cref{lem:clock}. All large rows and
old dangerous gradients are blocked by that same covariance. Therefore
the lifetime bound holds through the new endpoint, including a first
crossing or a simultaneous retirement. Normalization archives completed
intervals and starts new ones at that endpoint, so it preserves the row
identity and increasing labels.

For an old live record that survives below threshold, the cap is inactive
and $X^+=Xe^{\lambda_k\Delta Z}=e^{\lambda_k Z^+}<e^{L_k}$.
At a surviving crossing its exponential is exactly $e^{L_k}$ and its new
status is dangerous. An old dangerous record retains its mark and
exponential value while its regularized discrepancy cannot increase. A retiring
record keeps the final exponential value just computed. Every new birth has
$Z=0$ and $X=1$. These statements establish the record conditions for
every provisional successor.

For each $j,k$, \cref{thm:local-estimate} applies to the actual
complete update made from the old live records. The predecessor value
$G^{(j,k)}$ is fixed, so
\[
 \E\bigl[w_k e^{G^{+(j,k)}/K_k}\bigr]
 =w_ke^{G^{(j,k)}/K_k}
       \E e^{(G^{+(j,k)}-G^{(j,k)})/K_k}
 \le w_ke^{G^{(j,k)}/K_k}.
\]
Sum over the same finite column and class sets to obtain
$\E\Psi^+\le\Psi$. At least one positive-weight atom has
$\Psi^+\le\Psi$. Its complete successor already has every other
inductive property. The squared-norm identity holds for that atom as well.
\end{proof}

All the column--class estimates use the same direction law at this
predecessor. We sum their expected contributions before selecting the
atom. Individual quantities $G^{(j,k)}$ may increase on the selected
step; the bound on their aggregate potential $\Psi$ controls the
dangerous-family count at the next state. The geometric and record identities
already hold for every provisional successor, so the chosen atom
preserves the complete invariant.

\Needspace{14\baselineskip}
\begin{theorem}[Finite prescribed-depth trajectory]\label{thm:trajectory}
If $0<\gamma<\min\{1/2,d_b\}$ and
$\gamma\le h_{\mathrm{an}}$, there is a finite trajectory starting
at $b$ and ending at a state $x$ with at most $N_R(q)$ alive
coordinates. Its number of steps is at most
\[
 \left\lfloor\frac{q-\norm b_2^2}{\gamma^2}\right\rfloor.
\]
At its endpoint,
\begin{equation}\label{eq:raw-endpoint}
 \norm{A(x-b)}_\infty
 \le\frac72\sum_{k=0}^{K}B_k(R)+2+21\gamma.
\end{equation}
\end{theorem}
\begin{proof}
Initialize by \cref{lem:initialization}. At every state with more
than $N_R(q)$ alive coordinates, choose the successor from
\cref{prop:successor}. After $t$ steps,
\[
 \norm{x_t}_2^2=\norm b_2^2+t\gamma^2\le q.
\]
A trajectory that had not stopped after the displayed integer number
of steps would admit one more step by the same proposition, contradicting
this inequality. Thus a stopping state exists within that bound.
Apply \cref{cor:row-bound} to its actual history to obtain
\eqref{eq:raw-endpoint}.
\end{proof}

At this endpoint, round each buffered frozen coordinate to its sign and
leave each alive coordinate unchanged. Call the resulting point $y$.
The number of fractional coordinates of $y$ is at most $N_R(q)$.
Every altered coordinate changes by at most $\gamma$, and
$|a_i(j)|\le1$ by the column norm assumption. Hence
\begin{equation}\label{eq:buffered-rounding}
 \norm{A(y-b)}_\infty
 \le\frac72\sum_{k=0}^{K}B_k(R)+2+(q+21)\gamma.
\end{equation}
The original fixed signs, omitted from the working matrix, are then
restored without changing this increment.

\subsection{The threshold sum and finite attainment}
\label{sec:finite-attainment}
\begin{proposition}[Threshold sum]\label{prop:threshold-sum}
For the schedule in \eqref{eq:thresholds},
\begin{equation}\label{eq:threshold-sum}
 \sum_{k=0}^{\infty}B_k(0)
 <\frac{683637613}{10^6}<\frac{17091}{25}.
\end{equation}
For every $R\ge0$ and every finite cutoff,
\begin{equation}\label{eq:depth-sum}
 \sum_{k=0}^{K}B_k(R)
 \le\frac{17091}{25}(1+R)^{1/4}.
\end{equation}
\end{proposition}
\begin{proof}
The rational certificate for \eqref{eq:threshold-sum} is proved in
\cref{app:certificate}. Since $L_k>1$,
$L_k(L_k+R)\le L_k^2(1+R)$. Thus the fourth-root summand of
$B_k(R)$ is at most its depth-zero value times $(1+R)^{1/4}$.
The square-root summand is independent of $R$ and satisfies the same
bound because $(1+R)^{1/4}\ge1$. Sum the resulting inequality.
\end{proof}

The numerical bound leaves a positive reserve:
\begin{equation}\label{eq:numerical-reserve}
 \frac72\frac{17091}{25}+2
 =2395-\frac{13}{50}.
\end{equation}
It pays for the buffered rounding and the finite-mesh overshoot without
passing to a limiting trajectory.

\begin{proof}[Proof of \cref{thm:partial}]
If $q\le8$, take $y=b$. Suppose $q>8$ and use the working matrix
on its fractional columns. Fix $R$ and the finite class schedule first.
Choose
\begin{equation}\label{eq:final-mesh}
 \gamma=\min\left\{\frac{d_b}{2},\frac{h_{\mathrm{an}}}{2},
                    \frac{13}{50(q+21)}\right\}>0.
\end{equation}
The first two terms imply the hypotheses of \cref{thm:trajectory}.
By \eqref{eq:buffered-rounding}, \cref{prop:threshold-sum}, and
\eqref{eq:numerical-reserve}, with $z=(1+R)^{1/4}\ge1$,
\begin{align*}
 \norm{A(y-b)}_\infty
 &\le\frac72\frac{17091}{25}z+2+(q+21)\gamma\\
 &\le\left(2395-\frac{13}{50}\right)z+\frac{13}{50}
 \le2395z.
\end{align*}
The stopping rule gives the stated fractional count, and restoring the
omitted coordinates preserves every original sign.
\end{proof}

\subsection{Euclidean completion}\label{sec:completion}
\begin{lemma}\label{lem:euclidean-rounding}
Let $A\in\R^{m\times n}$ have columns of Euclidean norm at most one,
and let $y\in[-1,1]^n$ have at most $r$ fractional coordinates. There is
a signing $\sigma$ preserving its fixed coordinates such that
\begin{equation}\label{eq:euclidean-rounding}
 \norm{A(\sigma-y)}_2\le\sqrt r,
 \qquad \norm{A(\sigma-y)}_\infty\le\sqrt r.
\end{equation}
\end{lemma}
\begin{proof}
Round the fractional coordinates independently, choosing
$\sigma(j)\in\{-1,1\}$ with mean $y(j)$. The centered coordinate
increments are independent, so their mixed second moments vanish. Thus
\[
 \E\norm{A(\sigma-y)}_2^2
 =\sum_{j\text{ fractional}}(1-y(j)^2)\norm{A_{\cdot j}}_2^2
 \le r.
\]
The finite average has an outcome with squared norm at most $r$.
The $\ell_\infty$ estimate follows from the Euclidean estimate.
\end{proof}

\begin{proof}[Proof of \cref{cor:completion,thm:main}]
For $q>8$, take $R=\logplus(q/9)$. Since $q$ is an integer, it
satisfies $q\ge9$ and $qe^{-R}=9$, so $N_R(q)=8$.
Apply \cref{thm:partial}, followed by \cref{lem:euclidean-rounding}
with $r=8$. This gives
\[
 \norm{A(\sigma-b)}_\infty
 \le2395\left(1+\logplus\frac q9\right)^{1/4}+2\sqrt2.
\]
If $q\le8$, direct Euclidean rounding gives the stronger bound
$\sqrt q$. In both cases existing signs are preserved.
Taking $b=0$ gives \cref{thm:main}, with $q=n$. An empty set of
columns has discrepancy zero.
\end{proof}

\appendix
\section{A uniform weighted fourth-derivative estimate}\label[appendix]{app:derivatives}
This appendix proves \cref{lem:weighted-fourth}, including explicit
constants. The weights may vanish and the weighted energy may be zero.
The estimate therefore applies uniformly to the live-record families
encountered in the construction.

Put
\begin{align*}
 D_0&=P_0+Q_0,&
 h_b&=(1+D_0+3\sigma D_0)^{-1},&
 A_0&=P_0+2Q_0,\\
 L_2&=2+8Q_0,&
 m_1&=3\sigma A_0,&
 m_2&=3\sigma(A_0^2+2Q_0),\\
 b_2&=3(L_2+2),&
 b_3&=3A_0(L_2+6),&
 b_4&=3\bigl((A_0^2+12Q_0)L_2+12Q_0\bigr),\\
 c_1&=9\sigma L_2.
\end{align*}
One admissible constant is
\begin{equation}\label{eq:C4-explicit}
 C_4=3\bigl(b_4+4m_1b_3+3m_2b_2+6m_1^2b_2+c_1m_1^2\bigr).
\end{equation}
All constants are nonnegative and $h_b>0$.

\begin{proof}[Proof of \cref{lem:weighted-fourth}]
For one index, abbreviate $z(t)=pt-qt^2$ and $a(t)=p-2qt$.
On $|t|\le h_b\le1$,
\begin{equation}\label{eq:inner-basic}
 |z(t)|\le D_0|t|<1,\qquad |a(t)|\le A_0,
 \qquad a(t)^2\le L_2(p^2+q).
\end{equation}
The last inequality follows from
$a(t)^2\le2p^2+8q^2t^2\le2p^2+8Q_0q$.
The mean value theorem and $e^{|z|}<3$ give
\[
 |J(t)|\le3\sum_a\omega_a|p_at-q_at^2|
 \le3\sigma D_0|t|<1.
\]
Thus $e^{z(t)}<3$ and $e^{J(t)}<3$ throughout this interval.
The statements remain valid when $D_0=0$, in which case the inner
exponents are identically zero.

Successive differentiation gives
\begin{align}
 (e^z)'&=ae^z,&
 (e^z)''&=(a^2-2q)e^z,\notag\\
 (e^z)'''&=(a^3-6qa)e^z,&
 (e^z)^{(4)}&=(a^4-12qa^2+12q^2)e^z.
 \label{eq:inner-derivatives}
\end{align}
Write $W=\sum_a\omega_a(p_a^2+q_a)$. Using
\eqref{eq:inner-basic} in \eqref{eq:inner-derivatives} yields
\begin{align}
 |J'|&\le m_1,& |J''|&\le m_2,\notag\\
 |J''|&\le b_2W,& |J'''|&\le b_3W,&
 |J^{(4)}|&\le b_4W.
 \label{eq:inner-weighted-bounds}
\end{align}
For example, $|a^3-6qa|\le A_0(L_2+6)(p^2+q)$,
and
\[
 |a^4-12qa^2+12q^2|
 \le\bigl((A_0^2+12Q_0)L_2+12Q_0\bigr)(p^2+q).
\]
A weighted Cauchy--Schwarz inequality gives an additional estimate:
\begin{equation}\label{eq:first-derivative-square}
 |J'|^2
 \le\left(\sum_a\omega_a\right)
          \sum_a\omega_a a_a(t)^2e^{2z_a(t)}
 \le9\sigma L_2W=c_1W.
\end{equation}

The fourth derivative of the outer exponential is
\begin{equation}\label{eq:outer-fourth}
 (e^J)^{(4)}
 =e^J\bigl(J^{(4)}+4J'J'''+3(J'')^2
                     +6(J')^2J''+(J')^4\bigr).
\end{equation}
Bound the five terms in order by
\[
 b_4W,\quad 4m_1b_3W,\quad 3m_2b_2W,\quad
 6m_1^2b_2W,\quad c_1m_1^2W.
\]
The factor $e^J<3$ proves \eqref{eq:weighted-fourth} with
\eqref{eq:C4-explicit}. Every estimate is proportional to $W$;
there is no division by it.
\end{proof}

For the parameters used in \cref{sec:local}, $P_0,Q_0>0$ and
$\sigma>0$, so $C_4>0$ and the positive mesh in
\eqref{eq:analytic-mesh} is well defined. The constants depend only on
the chosen class, the prescribed depth, and the fixed geometric
parameters. They are independent of the number of live records, their
individual coefficients, and the eigenvalues of the selected covariance.

\section{The rational threshold-sum certificate}\label[appendix]{app:certificate}
We give the exact certificate used in \cref{prop:threshold-sum}. Its
logarithm enclosures and root comparisons require only rational
arithmetic, and the infinite remainder is bounded by two geometric
series. The complete integer data appear in \cref{tab:roots}.

Write $B_k(0)=a_k+b_k$ for the fourth-root and square-root terms in
\eqref{eq:thresholds}. The exact logarithmic parameters are
\begin{equation}\label{eq:logarithm-identities}
\begin{aligned}
 L_0&=\log\frac{125440}{2193},&\qquad
 L_1&=\log\frac{2132480}{15867},\\
 d&=\log\frac{34}{15},&
 L_k&=L_1+(k-1)d\quad(k\ge1).
\end{aligned}
\end{equation}

\subsection{Rational logarithm enclosures}
For $1\le y\le2$, set $u=(y-1)/(y+1)$ and define
\[
 S_{16}(u)=2\sum_{t=0}^{15}\frac{u^{2t+1}}{2t+1},\qquad
 E_{16}(u)=\frac{2u^{33}}{33(1-u^2)}.
\]
The positive series for $\log y=2\operatorname{arctanh}u$ gives
\begin{equation}\label{eq:logarithm-enclosure}
 S_{16}(u)\le\log y\le S_{16}(u)+E_{16}(u).
\end{equation}
Indeed, every denominator in the omitted terms is at least $33$, so
their sum is bounded by the stated geometric series. For rational
$x\ge1$, write $x=2^e y$ with integer $e\ge0$ and $1\le y<2$.
Add $e$ copies of the corresponding bound for $\log2$, using $u=1/3$.
This proves the following enclosures, whose displayed decimals are exact
rationals:
\begin{center}
\begin{tabular}{@{}lrr@{}}
\toprule
Quantity & Lower bound & Upper bound\\
\midrule
$L_0$ & $4.046557087166$ & $4.046557087167$\\
$L_1$ & $4.900799419906$ & $4.900799419907$\\
$d$   & $0.818310323513$ & $0.818310323514$\\
\bottomrule
\end{tabular}
\end{center}
The verification of these enclosures consists of substituting each of
the three rational arguments into \eqref{eq:logarithm-enclosure} and
comparing fractions. Denote the resulting bounds by
$L_0^\pm,L_1^\pm,d^\pm$, and put
$L_k^\pm=L_1^\pm+(k-1)d^\pm$ for $k\ge1$.

\subsection{The finite root comparisons}
Let $g=10^6$ and $N=128$. For $0\le k\le N$, choose positive
integers $A_k,D_k$ satisfying
\begin{align}
 \left(\frac{A_k}{g}\right)^4
 &>\frac{1925648}{675}\frac{s_k(L_k^+)^2}{\delta_k},&
 \left(\frac{D_k}{g}\right)^2
 &>\frac{244}{9}s_kL_k^+.
 \label{eq:root-comparisons}
\end{align}
Then $a_k<A_k/g$ and $b_k<D_k/g$. The integers in
\cref{tab:roots} are the least choices for these rational upper
radicands. For a nonnegative rational $v$ and degree $\ell\in\{2,4\}$,
the least integer with $(a/g)^\ell>v$ is
\begin{equation}\label{eq:integer-root-construction}
 a=1+\left\lfloor\sqrt[\ell]{\left\lfloor g^\ell v\right\rfloor}
       \right\rfloor.
\end{equation}
Integer square roots, applied twice when $\ell=4$, evaluate this
formula exactly. Checking the displayed certificate only requires
raising each integer to the indicated power and comparing fractions.
The additional lower comparisons for $A_k-1,D_k-1$ verify minimality,
although only the upper comparisons are needed for the proof.

The table gives
\begin{equation}\label{eq:prefix-sums}
 \sum_{k=0}^{127}A_k=557621898,\qquad
 \sum_{k=0}^{127}D_k=126015679,\qquad A_{128}=4,\quad D_{128}=1.
\end{equation}
The complete certificate in machine-readable form and an exact rational
verifier are included as ancillary files. The same root comparisons and finite sums
are proved in the Lean project.

\Needspace{25\baselineskip}
\subsection{The two infinite tails}
For $k\ge1$, the schedule gives the exact ratios
\begin{equation}\label{eq:tail-ratios}
 \left(\frac{a_{k+1}}{a_k}\right)^4
 =\frac{17}{30}\left(1+\frac d{L_k}\right)^2,\qquad
 \left(\frac{b_{k+1}}{b_k}\right)^2
 =\frac12\left(1+\frac d{L_k}\right).
\end{equation}
They decrease with $k$, since $L_k$ increases. At $N=128$, rational
comparison using the logarithm enclosures gives
\begin{equation}\label{eq:certified-tail-ratios}
 \frac{17}{30}\left(1+\frac{d^+}{L_{128}^-}\right)^2
 <\left(\frac78\right)^4,\qquad
 \frac12\left(1+\frac{d^+}{L_{128}^-}\right)
 <\left(\frac34\right)^2.
\end{equation}
Therefore the two tails, each starting at index $128$, are bounded by
$8a_{128}$ and $4b_{128}$. Equations
\eqref{eq:root-comparisons} and \eqref{eq:prefix-sums} yield
\begin{align*}
 \sum_{k=0}^{\infty}B_k(0)
 &<\frac{557621898+126015679+8\cdot4+4\cdot1}{10^6}\\
 &=\frac{683637613}{10^6}<\frac{17091}{25}.
\end{align*}
This proves \eqref{eq:threshold-sum}, including the terms beyond every
possible finite cutoff.

\begin{table}[p]
\centering
\caption{Integer upper bounds with grid $g=10^6$. The first $128$ pairs form the prefix, and the pair at index $128$ starts both tail bounds.}
\label{tab:roots}
\small
\setlength{\tabcolsep}{5pt}
\renewcommand{\arraystretch}{1.04}
\begin{tabular*}{\textwidth}{@{\extracolsep{\fill}}rrr@{\hspace{2em}}rrr@{\hspace{2em}}rrr@{}}
\toprule
$k$ & $A_k$ & $D_k$ & $k$ & $A_k$ & $D_k$ & $k$ & $A_k$ & $D_k$\\
\cmidrule(lr){1-3}\cmidrule(lr){4-6}\cmidrule(lr){7-9}
0 & 49767931 & 30665010 & 43 & 348827 & 33 & 86 & 1072 & 1\\
1 & 47948321 & 23862654 & 44 & 305788 & 24 & 87 & 935 & 1\\
2 & 44940310 & 18227810 & 45 & 268004 & 17 & 88 & 816 & 1\\
3 & 41687607 & 13780294 & 46 & 234841 & 12 & 89 & 711 & 1\\
4 & 38366183 & 10336017 & 47 & 205742 & 9 & 90 & 621 & 1\\
\addlinespace[2pt]
5 & 35090213 & 7704487 & 48 & 180216 & 6 & 91 & 541 & 1\\
6 & 31932737 & 5714088 & 49 & 157828 & 5 & 92 & 472 & 1\\
7 & 28938804 & 4220311 & 50 & 138198 & 4 & 93 & 412 & 1\\
8 & 26134186 & 3106175 & 51 & 120989 & 3 & 94 & 359 & 1\\
9 & 23531329 & 2279379 & 52 & 105907 & 2 & 95 & 313 & 1\\
\addlinespace[2pt]
10 & 21133501 & 1668379 & 53 & 92690 & 2 & 96 & 273 & 1\\
11 & 18937723 & 1218440 & 54 & 81111 & 1 & 97 & 238 & 1\\
12 & 16936864 & 888101 & 55 & 70968 & 1 & 98 & 208 & 1\\
13 & 15121141 & 646200 & 56 & 62085 & 1 & 99 & 181 & 1\\
14 & 13479199 & 469461 & 57 & 54306 & 1 & 100 & 158 & 1\\
\addlinespace[2pt]
15 & 11998875 & 340588 & 58 & 47496 & 1 & 101 & 138 & 1\\
16 & 10667750 & 246783 & 59 & 41534 & 1 & 102 & 120 & 1\\
17 & 9473526 & 178611 & 60 & 36317 & 1 & 103 & 105 & 1\\
18 & 8404287 & 129137 & 61 & 31751 & 1 & 104 & 91 & 1\\
19 & 7448672 & 93279 & 62 & 27756 & 1 & 105 & 80 & 1\\
\addlinespace[2pt]
20 & 6595978 & 67319 & 63 & 24261 & 1 & 106 & 70 & 1\\
21 & 5836218 & 48545 & 64 & 21204 & 1 & 107 & 61 & 1\\
22 & 5160147 & 34981 & 65 & 18530 & 1 & 108 & 53 & 1\\
23 & 4559259 & 25189 & 66 & 16191 & 1 & 109 & 46 & 1\\
24 & 4025772 & 18127 & 67 & 14147 & 1 & 110 & 40 & 1\\
\addlinespace[2pt]
25 & 3552593 & 13037 & 68 & 12359 & 1 & 111 & 35 & 1\\
26 & 3133287 & 9371 & 69 & 10796 & 1 & 112 & 31 & 1\\
27 & 2762032 & 6733 & 70 & 9430 & 1 & 113 & 27 & 1\\
28 & 2433576 & 4835 & 71 & 8236 & 1 & 114 & 23 & 1\\
29 & 2143194 & 3470 & 72 & 7193 & 1 & 115 & 21 & 1\\
\addlinespace[2pt]
30 & 1886644 & 2490 & 73 & 6281 & 1 & 116 & 18 & 1\\
31 & 1660126 & 1786 & 74 & 5485 & 1 & 117 & 16 & 1\\
32 & 1460240 & 1280 & 75 & 4789 & 1 & 118 & 14 & 1\\
33 & 1283952 & 918 & 76 & 4181 & 1 & 119 & 12 & 1\\
34 & 1128556 & 657 & 77 & 3650 & 1 & 120 & 11 & 1\\
\addlinespace[2pt]
35 & 991640 & 471 & 78 & 3186 & 1 & 121 & 9 & 1\\
36 & 871063 & 337 & 79 & 2781 & 1 & 122 & 8 & 1\\
37 & 764919 & 242 & 80 & 2427 & 1 & 123 & 7 & 1\\
38 & 671519 & 173 & 81 & 2118 & 1 & 124 & 6 & 1\\
39 & 589364 & 124 & 82 & 1849 & 1 & 125 & 6 & 1\\
\addlinespace[2pt]
40 & 517127 & 89 & 83 & 1613 & 1 & 126 & 5 & 1\\
41 & 453631 & 64 & 84 & 1408 & 1 & 127 & 4 & 1\\
42 & 397838 & 46 & 85 & 1228 & 1 & 128 & 4 & 1\\
\bottomrule
\end{tabular*}
\end{table}
\clearpage

\section{The Lean formalization}\label[appendix]{app:formalization}
The Lean project proves the endpoint theorems in
\cref{sec:introduction} from the explicit parameter
\texttt{BansalJiangA4}, the covariance statement in \cref{thm:source}.
The formal derivation includes its finite-row consequence, the historical
states and finite path, the numerical schedule bound, and final rounding.

\subsection{The final declarations and their scope}
The file \path{KomlosQuarter/R37/FinalTheorem.lean} contains the
following declarations in the namespace \texttt{KomlosQuarter.R37}:
\begin{center}
\small
\renewcommand{\arraystretch}{1.12}
\begin{tabularx}{\textwidth}{@{}>{\raggedright\arraybackslash}p{0.36\textwidth}>{\raggedright\arraybackslash}X@{}}
\toprule
Declaration & Mathematical statement\\
\midrule
\texttt{partial\_coloring\_2395} &
The prescribed-depth endpoint in \cref{thm:partial}, preserving
original signs and controlling the number of fractional coordinates.\\[3pt]
\texttt{full\_coloring\_2395} &
The arbitrary-start signing bound in \cref{cor:completion}.\\[3pt]
\texttt{komlos\_quarter\_2395} &
The zero-start matrix discrepancy bound in \cref{thm:main}.\\
\bottomrule
\end{tabularx}
\end{center}
Matrices have arbitrary finite row and column index types, real entries,
and $\sum_i A(i,j)^2\le1$ for each column. The partial endpoint uses
nested square roots, proved equal to the real power $x^{1/4}$ for
$x\ge0$; the full statement uses that power. The stopping target is
the natural-number form of $\max\{\lceil qe^{-R}\rceil-1,8\}$,
including the empty fractional face.

For an arbitrary starting point, A.4 is assumed on every alive
coordinate space of its fractional face. The zero-start declaration
instantiates the universal finite-dimensional A.4 statement on those
spaces. A.4 is the sole external research theorem parameter; the
remaining ingredients are proved within the project.

\subsection{Correspondence with the written proof}
The principal groups of modules are listed below.\\
Module names are
relative to \texttt{KomlosQuarter/R37/}.\par
{\small
\renewcommand{\arraystretch}{1.10}
\begin{longtable}{@{}>{\raggedright\arraybackslash}p{0.36\textwidth}>{\raggedright\arraybackslash}p{\dimexpr0.64\textwidth-2\tabcolsep\relax}@{}}
\toprule
Part of the paper & Principal modules\\
\midrule
\endfirsthead
\toprule
Part of the paper & Principal modules\\
\midrule
\endhead
\bottomrule
\endfoot
\bottomrule
\endlastfoot
The source theorem and finite-row reduction &
\path{SDPSource.lean}, \path{AliveSDP.lean},
\path{CurrentSDP.lean}, \path{LiveSDP.lean}\\[4pt]
The analytic bound for actual historical records &
\path{StateStep.lean}, \path{SDPStep.lean}, and their signed
Taylor and covariance dependencies\\[4pt]
Dangerous counts and preservation of the schedule invariant &
\path{ScheduledCounts.lean}, \path{ScheduledStep.lean},
\path{ScheduledInvariant.lean}\\[4pt]
Initialization, finite paths, and the stopping count &
\path{GivenInitialization.lean}, \path{GivenMesh.lean},
\path{GivenPath.lean}, \path{TrackedPath.lean},
\path{StoppingDepth.lean}\\[4pt]
The row identity and the endpoint estimate &
\path{SectionLedger.lean}, \path{SectionStep.lean},
\path{PeelBudget.lean}, \path{LedgerMemory.lean},
\path{CoreLedger.lean}, \path{NumericalLedger.lean}\\[4pt]
The rational schedule and finite reserve &
\path{ScheduleLogs.lean}, \path{ScheduleMajorants.lean},
\path{RootCertificate.lean}, \path{ScheduleTail.lean},
\path{FiniteAttainment.lean}\\[4pt]
Return to the ambient face and final rounding &
\path{BufferedRounding.lean}, \path{AmbientEndpoint.lean},
\path{EuclideanRounding.lean}, \path{FinalTheorem.lean}\\
\end{longtable}
}

\Needspace{9\baselineskip}
Two implementation choices account for differences in the intermediate
formulas. The formal finite-row reduction may repeat a
block more times than the least choice in \cref{cor:finite-row};
testing on repeated vectors recovers the same coefficient. The formal
schedule uses $q+4$ classes, indexed from zero, which is the sufficient
cutoff $K=q+3$ admitted in \cref{sec:restrictions}. Every estimate
above is uniform over that finite cutoff. The formal mesh is chosen
from sufficient positive upper bounds; its proof uses exactly the
reserve in \eqref{eq:numerical-reserve}.

\subsection{Reproduction}
The project is pinned to Lean \texttt{4.33.1} and to the mathlib revision
\begin{center}
\texttt{0df444a360eaa60ab8c11dca51a86af692955474}.
\end{center}
The ancillary directory contains the source files, the toolchain and
package manifests, an axiom-query file, and build instructions. The
final declarations have the ordinary logical dependencies
\texttt{propext}, \texttt{Classical.choice}, and \texttt{Quot.sound};
A.4 remains an explicit theorem hypothesis.

The arithmetic program \path{finite_attainment.py} verifies the
certificate in \cref{app:certificate} using integers and rational
fractions. The Lean project includes the same schedule bound in the
formal derivation of the endpoint theorems.

\section*{Use of computational tools}
ChatGPT and Codex (OpenAI) were used in developing the argument, writing and
refining the Lean formalization and exact-arithmetic code, and preparing the
manuscript.
{\small
\bibliographystyle{alpha}
\bibliography{references}
}
\end{document}